\documentclass{birkjour_t2}
\usepackage{graphicx}
\usepackage[caption=false]{subfig}
\usepackage[dvipsnames]{color}
\usepackage{amsmath,amssymb,amsfonts,amsthm}
\usepackage{cancel}
\usepackage{hyperref}

\newtheorem{theorem}{Theorem}[section]
\newtheorem{remark}[theorem]{Remark}
\newtheorem{corollary}[theorem]{Corollary}
\newtheorem{lemma}[theorem]{Lemma}

\newtheorem{assumption}[theorem]{Assumption}
\newtheorem{definition}[theorem]{Definition}

\begin{document}

\title[Time-fractional Fokker--Planck equation with general forcing]{Existence, 
uniqueness and regularity of the\\ solution of the 
time-fractional Fokker--Planck equation\\ with general forcing}

\footnote{Kim-Ngan Le: School of Mathematical Sciences, Monash University, VIC 3800, Australia. \\
William McLean: School of Mathematics and Statistics, The University of New South Wales, Sydney 2052, Australia.\\
Martin Stynes: Applied and Computational Mathematics Division, Beijing 
Computational Science Research Center, Beijing 100193, China.}

\author{Kim-Ngan Le}
\address{School of Mathematical Sciences, Monash University, VIC 3800, Australia. ORCID 0000-0002-7628-9379. The research of this author is supported in part by  the Australian Government through the Australian Research Councils Discovery Projects funding scheme (project number DP170100605).}
\email{ngan.le@monash.edu}

\author{William McLean }
\address{School of Mathematics and Statistics. 
The University of New South Wales, Sydney 2052, Australia.   ORCID 0000-0002-7133-2884.}
\email{w.mclean@unsw.edu.au}

\author{Martin Stynes}
\address{Applied and Computational Mathematics Division, Beijing 
Computational Science Research Center, Beijing 100193, China. ORCID 0000-0003-2085-7354.   The research of this author is supported in part by the National Natural Science Foundation of China under grants 91430216 and NSAF-U1530401. Corresponding author; telephone +86-188-0011-8417.}
\email{m.stynes@csrc.ac.cn}

\begin{abstract}
A time-fractional Fokker--Planck initial-boundary value problem is considered, 
with  differential operator 
$u_t-\nabla\cdot(\partial_t^{1-\alpha}\kappa_\alpha\nabla u 
-\textbf{F}\partial_t^{1-\alpha}u)$, where $0<\alpha <1$. The forcing function 
$\textbf{F} = \textbf{F}(t,x)$, which is more difficult to analyse than the case 
$\textbf{F}=\textbf{F}(x)$ investigated previously by other authors. The 
spatial domain $\Omega \subset\mathbb{R}^d$, where  $d\ge 1$, has a smooth
boundary. Existence, uniqueness and regularity of a mild solution $u$ is proved 
under the hypothesis that the initial data $u_0$ lies in $L^2(\Omega)$. For 
$1/2<\alpha<1$ and $u_0\in H^2(\Omega)\cap H_0^1(\Omega)$, it is shown that $u$ 
becomes a classical solution of the problem. Estimates of time derivatives of 
the classical solution are derived---these are known to be needed in numerical 
analyses of this problem. 
\end{abstract}


			
	
\maketitle

\section{Introduction}\label{sec:intro}
In this paper, we study the existence, uniqueness and regularity of 
solutions to the following inhomogeneous, time-fractional Fokker--Planck 
initial-boundary value problem:
\begin{subequations}\label{prob}
\begin{align}
u_t(t,x)-\nabla\cdot(\partial_t^{1-\alpha}\kappa_\alpha\nabla u
    -\textbf{F}\partial_t^{1-\alpha}u)(t,x)&=g(t,x) \ 
\text {for }(t,x)\in (0,T)\times\Omega,  \label{proba} \\
u(0,x)&=u_0(x)\ \text{ for } x\in\Omega,  \label{probb} \\
u(t,x)&=0 \ \text{ for } x\in\partial\Omega \text{ and } 0<t<T, \label{probc}
\end{align}
\end{subequations}
where $\kappa_\alpha>0$ is constant and 
 $\Omega$ is an open bounded domain with $C^2$ boundary in~$\mathbb{R}^d$ for some~$d\geq 1$.
In~\eqref{proba}, one has $0<\alpha<1$ and $\partial_t^{1-\alpha}$ is the 
standard Riemann--Liouville fractional derivative operator defined by
$\partial_t^{1-\alpha}u=(J^\alpha u)_t$, where $J^\beta$ denotes the
Riemann--Liouville fractional integral operator of order~$\beta$, viz.,
\[
J^{\beta} u=\int_0^t \omega_{\beta}(t-s)u(s)ds
    \ \text{ where }\ 
\omega_\beta(t):=\frac{t^{\beta-1}}{\Gamma(\beta)} \text{ for }\beta >0.
\]
Regularity hypotheses on $\textbf{F}, g$~and $u_0$ will be imposed later.

The problem~\eqref{prob} was considered in~\cite{HLeS17,LeMM16,PS17}. 
We describe it as ``general forcing" since $\textbf{F} = \textbf{F}(t,x)$;  
this is a more difficult problem than the special case where $\textbf{F} = 
\textbf{F}(x)$, which can be reduced to a problem already studied by several 
authors (see, e.g.,~\cite{EIDELMAN2004,LUCHKO2009,LUCHKO2010,MU2017,WANG2012}). 
More precisely, when the force~$\textbf{F}$ may depend on $t$ as well as $x$, 
equation~\eqref{prob} cannot be rewritten in the form of the fractional evolution equation
\begin{equation}\label{eq:fr sub}
J^{1-\alpha}(u_t) + A u 
= h(t,u,\nabla u, g,\textbf{F}),
\end{equation}
in which the first term is a Caputo fractional derivative, the 
operator~$A=-\kappa_\alpha \Delta $, and the function $h$ does not depend 
explicitly on~$\partial_t^{1-\alpha}u$.

The regularity of the solution to the Cauchy problem for~\eqref{eq:fr sub} was 
studied in~\cite{EIDELMAN2004}; there a fundamental solution of that problem was 
constructed and investigated for a more general evolution equation where the 
operator $A$ in~\eqref{eq:fr sub} is a uniformly elliptic operator with variable 
coefficients that acts on the spatial variables. The Cauchy problem was also 
considered in~\cite{MU2017}  where~$h= h(t,u,g,\textbf{F})$ lies in a space of 
weighted H\"older continuous functions, and in~\cite{WANG2012} for the case 
where $A$~is almost sectorial. Existence and uniqueness of a solution to the 
initial-boundary value problem where \eqref{proba} is replaced by~\eqref{eq:fr 
sub} is shown in~\cite{LUCHKO2009,LUCHKO2010}. 

To the best of our knowledge, the well-posedness and regularity properties 
of solutions to~\eqref{prob} are open questions at present, apart from a 
recent preprint~\cite{MMAK2018} which treats a wider class of problems that 
includes \eqref{prob} as a special case.  The analysis in \cite{MMAK2018} 
proceeds along broadly similar lines to here---relying on Galerkin 
approximation, a fractional Gronwall inequality and compactness arguments---but 
employs a different sequence of \emph{a priori} estimates and does not make use 
of the weighted $L^2$-norm of Definition~\ref{def:weightnorm} or the 
Aubin--Lions--Simon lemma 
(Lemma~\ref{lem: Aubin}).  An interesting consequence of the approach taken here 
is that the constants in our estimates remain bounded as~$\alpha\to1$, which one 
expects since in this limit \eqref{prob} becomes the classical Fokker--Planck 
equation.  However, the estimates in sections \ref{sec:classical}~and 
\ref{sec:regclassical} are valid only for~$1/2<\alpha<1$, with constants that 
blow up as~$\alpha\to1/2$ (cf.~the comment following Assumption~\ref{ass:12}).  
By contrast, the results in~\cite{MMAK2018} hold for the full range of values 
$0<\alpha<1$, but with constants that blow up as~$\alpha\to1$.  Also, the 
analysis is significantly longer than the one presented here.

The main contributions of our work are:
\begin{itemize}
\item A proof in Theorem~\ref{the:exit0} of existence and uniqueness
of the mild solution of~\eqref{prob} for the case $\alpha\in (0,1)$ 
and $u_0\in L^2(\Omega)$;
\item By imposing a further condition on $u_0$ and restricting $\alpha$ to 
lie in $(1/2, 1)$, the mild solution becomes the classical solution 
of~\eqref{prob} described in Theorem~\ref{the:exit};
\item Estimates of time derivatives of the classical solution in 
Theorem~\ref{th: regu}.
\end{itemize}

The paper is organized as follows. Section~\ref{sec:definitions} introduces our 
basic notation and the definitions of mild and classical solutions 
of~\eqref{prob}. Various technical properties of fractional integral operators 
that will be used in our analysis are provided in 
Section~\ref{sec:preliminaries}. In Section~\ref{sec:Galerkin}, we introduce the 
Galerkin approximation of the solution of~\eqref{prob} and prove existence and 
uniqueness of approximate solutions.  Properties of the mild and classical 
solutions are derived in Sections~\ref{sec:mild} and~\ref{sec:classical}, 
respectively. Finally, in Section~\ref{sec:regclassical}, we provide estimates 
of the time derivatives of the classical solution in $L^2(\Omega)$~and 
$H^2(\Omega)$, needed for the error analysis of numerical methods for 
solving~\eqref{prob}; see, e.g.,~\cite{HLeS17,LeMM16,PS17}.

\section{Notation and definitions}\label{sec:definitions}
Throughout the paper, we often suppress the spatial variables and 
write $v$ or $v(t)$ instead of~$v(t,\cdot)$ for various functions~$v$.  We 
also use the notation $v'$ for the time derivative.
Let $\|\cdot\|$ denote the $L^2(\Omega)$ 
norm defined by $\|v\|^2=\langle v,v\rangle$, where $\langle \cdot, 
\cdot\rangle$ is the $L^2(\Omega)$ inner product. Let   $\|\cdot\|_{H^r(\Omega)}$ and 
$|\cdot|_{H^r(\Omega)}$  be the standard Sobolev norm and seminorm on the Hilbert space of 
functions whose $r$th-order derivatives lie in $L^2(\Omega)$. 
We borrow some standard notation from parabolic partial 
differential equations, e.g., $C([0,T]; L^2(\Omega))$.

Assume throughout the paper that the forcing function $\textbf{F} = (F_1, \dots, F_d)^T\in 
W^{1,\infty}((0,T)\times\Omega)$ and that  its divergence $\nabla\cdot\textbf{F}$ is continuous on 
$[0,T]\times\overline{\Omega}$. Then~\textbf{F} is continuous on 
$[0,T]\times\overline{\Omega}$ and we set
\[
\|\textbf{F}\|_\infty := \max_{1\leq i\leq d} 	\max_{(t,x)\in [0,T]\times\overline{\Omega}}     |F_i(t,x)|
\quad\text{and}\quad
\|\textbf{F}\|_{1,\infty} := \|\textbf{F}\|_\infty     +\max_{(x,t)\in [0,T]\times\overline{\Omega}}
	|\nabla\cdot\textbf{F}(t,x)|.
\]
Stronger assumptions on the regularity of \textbf{F} will be made in some sections.

We use~$C$ to denote a constant that  depends on the data  $\Omega, \kappa_\alpha, \textbf{F}$ 
and~$T$ of the problem~\eqref{prob} but is independent of any dimension of 
finite-dimensional spaces to be used in our Galerkin approximations. Here the 
unsubscripted 
constants $C$ are generic and can take different values in different places 
throughout the paper.

We now recall the definitions of some Banach spaces from \cite[p.301]{Evans10}:
\begin{definition}
Let $X$ be a real Banach space with norm $\|\cdot\|_X$. The space $C([0,T]; X)$ comprises all continuous functions $v:[0,T]\to X$ with
\[
\|v\|_{C([0,T]; X)} := \max_{0\le t\le T}\|v(t)\|_X.
\]
Let $p\in [1,\infty]$.
The space $L^p(0,T; X)$ comprises all measurable functions $v:[0,T]\to X$ for which 
\[
\|v\|_{L^p(0,T; X)} :=\begin{cases}
	\bigl(\int_0^T\|v(t)\|_X^p\,dt\bigr)^{1/p}<\infty
	&\text{when $1\le p<\infty$,} \\
	\operatorname{ess\, sup}_{0\le t\le T}\,\|v(t)\|_X< \infty 
    &\text{when $p = \infty$.}
    \end{cases}
\]
The space $W^{1,p}(0,T; X))$ comprises all measurable functions $v:[0,T]\to X$ for which 
\[
\|v\|_{W^{1,p}(0,T; X)} := \|v\|_{L^p(0,T; X)} + \|v'\|_{L^p(0,T; X)} \ \text{ is finite}.
\]
\end{definition}

Recall that $0<\alpha<1$.

\begin{definition}\label{def:weightnorm}
Given a  Banach space~$X$ with a norm (or seminorm) $\|\cdot\|_X$, define
$L^2_\alpha (0,T;X)$ to be the space of functions~$v:[0,T]\to X$ for which the 
following norm (or seminorm) is finite: 
\[
\|v\|_{L^2_\alpha (0,T; X)}  
	:= \max_{0\le t\le T} \left[ J^\alpha(\|v\|_X^2)(t) \right]^{1/2}
	= \max_{0\le t\le T} \left[  \frac1{\Gamma(\alpha)} \int_{s=0}^t (t-s)^{\alpha-1} \|v(s)\|_X^2  \, ds \right]^{1/2}.
\]
\end{definition}

For any Banach space $X$, clearly $L^2_\alpha (0,T;X) \subset L^2 (0,T;X)$ for $0<\alpha<1$ 
and $\|\cdot\|_{L^2_1 (0,T; X)} = \|\cdot\|_{L^2(0,T; X)}$ if we formally put $\alpha=1$ in Definition~\ref{def:weightnorm}.
For brevity, when $X=L^2(\Omega)$ we write
\[
\|v\|_{L_\alpha^2}=\|v\|_{L_\alpha^2(0,T;L^2(\Omega))}
\quad\text{and}\quad
\|v\|_{L^2}=\|v\|_{L^2(0,T;L^2(\Omega))}.
\]

The Mittag-Leffler function $E_{\alpha}(z)$ that is used in the  fractional Gronwall inequality of Lemma~\ref{lem:Gronwall} is defined by
\begin{equation*} 
E_{\alpha}(z):=\sum\limits_{k=0}^\infty \frac{z^k}{\Gamma(k\alpha+1)},
\end{equation*}
for $z\in \mathbb{R}$. Its properties  can be found in, e.g., \cite{Diet10}.

We now introduce the definitions of mild solutions and classical solutions to problem~\eqref{prob}. Set
\[
G(t):=u_0+\int_0^tg(s)\,ds \ \text{ for } 0\le t\le T.
\]

\begin{definition}[Mild solutions]\label{def:mild}
A mild solution of problem~\eqref{prob} is a function 
$u\in L^2(0,T;L^2(\Omega))$ such that
$J^\alpha u\in L^2\bigl(0,T;H^2(\Omega)\cap H^1_0(\Omega)\bigr)$ and $u$ 
satisfies 
\begin{equation}\label{prob0}
u-\kappa_\alpha \Delta(J^{\alpha} u)+\nabla\cdot(\mathbf{F}J^{\alpha}u)
-\nabla\cdot\biggl(\int_0^t \mathbf{F}'(s)J^{\alpha}u(s)\,ds\biggr)=G(t) \ \text{ a.e. on } (0,T)\times \Omega.
\end{equation}

\end{definition}

\begin{definition}[Classical solutions]\label{def:class}
A classical solution of problem~\eqref{prob} is a function 
$u$ belonging to the space $C([0,T];L^2(\Omega))\cap 
L^\infty(0,T;H^1_0(\Omega))\cap L^2(0,T;H^2(\Omega))$ such that
\[
 u'\in L^2(0,T;L^2(\Omega))\quad\text{and}\quad
\partial_t^{1-\alpha} u\in L^2(0,T;H^2(\Omega)),
\]
with $u$ satisfying~\eqref{proba} a.e. on $(0,T)\times \Omega$, and 
\eqref{probb} a.e. on~$\Omega$.
\end{definition}

\section{Technical preliminaries}\label{sec:preliminaries}
This section provides some  properties of fractional integrals that will be needed in our analysis. 

\begin{lemma}\label{lem:Gronwall} 
 \cite[Corollary 2]{YGD07}
Let $\beta>0$.  Assume that $a$~and $b$ are non-negative and non-decreasing 
functions on the interval~$[0,T]$, with $a\in L^1(0,T)$~and $b\in C[0,T]$. 
If $y\in L^1(0,T)$ satisfies
\[
0\le y(t)\le a(t)+b(t)\int_0^t\omega_\beta(t-s)y(s)\,ds
	\quad\text{for $0\le t\le T$,}
\]
then
\[
y(t)\le a(t)E_\beta\bigl(b(t)t^\beta\bigr)
\quad\text{for $0\le t\le T$.}
\] 
\end{lemma}

The following lemmas will be used several times in in our analysis. 

\begin{lemma}\cite[Lemma 2.2]{HLeS17} \label{lem:pd}
If $\alpha\in (1/2,1)$~and  $v(x,\cdot)\in L^2(0,T)$ for each 
$x\in\Omega$, then for $t\in [0,T]$,
\[
J^\alpha \langle J^\alpha v,v\rangle(t)
\geq \frac 1 2\|J^\alpha v(t)\|^2
\quad\text{and}\quad
\int_0^t\langle J^\alpha v,v\rangle(s)\,ds
\geq \frac 1 2J^{1-\alpha}(\|J^\alpha v\|^2)(t).
\]
\end{lemma}
%
\begin{lemma}\cite[Lemma 3.1 (ii)]{MS2014}\label{lem:pd2}
If $\alpha\in (0,1)$~and  $v(x,\cdot)\in L^2(0,T)$ for each $x\in\Omega$, 
then for $t\in [0,T]$,
\begin{equation*}
\int_0^t\langle J^\alpha v,v\rangle(s)\,ds
\geq \cos(\alpha\pi/2) \int_0^t\|J^{\alpha/2} v\|^2(s)\,ds.
\end{equation*}
\end{lemma}

\begin{lemma}\cite[Lemma 2.1]{HLeS17}\label{lem:minsk}
Let $\beta\in(0,1)$. If $\phi(\cdot,t)\in  L^2(\Omega)$ for $t\in[0,T]$, then
\[
\|J^\beta \phi(t)\|^2\leq\omega_{\beta+1}(t)\,J^\beta(\|\phi\|^2)(t)
	\quad\text{for $0\le t \le T$.}\label{lem:stability1}
\]
\end{lemma}
\begin{proof}
As the proof is short, we give it here for completeness.
The Cauchy--Schwarz inequality yields
\begin{align*}
\|J^\beta \phi(t)\|^2&=\int_\Omega \bigg[\int_0^t \omega_\beta(t-s)
	\phi(x,s)ds\,\bigg]^2\,dx
	\leq\int_\Omega
	\bigg[\int_0^t\omega_\beta(t-s)\,ds\bigg] 
	\bigg[\int_0^t\omega_\beta(t-s)\phi^2(x,s)\,ds\bigg]\,dx\\
	&=\omega_{\beta+1}(t)\int_0^t\omega_\beta(t-s)
		\int_\Omega \phi^2(x,s)\,dx\,ds 
	=\omega_{\beta+1}(t)\,J^\beta(\|\phi\|^2)(t).
\end{align*}
\end{proof}
\begin{lemma}\label{lem:L1}
For any $t>0$ and $\beta>0$,
\[
\|J^\beta\phi(t)\|\leq \frac{t^\beta}{\Gamma(\beta +1)}
	\|\phi\|_{L^\infty(0,t;L^2)},
\quad\text{for all }\phi\in L^\infty(0,t;L^2).
\]
If $\beta>1/2$, then
\[
\|J^\beta\phi(t)\|\leq \frac{t^{\beta-1/2}}{\Gamma(\beta)\sqrt{2\beta-1}}
	\|\phi\|_{L^2(0,t;L^2)}, \quad\text{for all }\phi\in L^2(0,t;L^2).
\]
\end{lemma}
\begin{proof}
Minkowski's integral inequality gives
\begin{align}\label{eq:L1}
\|J^\beta\phi(t)\|&=\left[\int_\Omega
	\left(\int_0^t\omega_\beta (t-s)\phi(s)\, ds\right)^2\,dx \right]^{1/2} 
	\leq  \int_0^t \omega_\beta (t-s)\|\phi(s)\|\, ds\\
&\leq \|\phi\|_{L^\infty(0,t;L^2)} \int_0^t \omega_\beta (t-s)\, ds 
= \frac{t^\beta}{\Gamma(\beta+1)}\|\phi\|_{L^\infty(0,t;L^2)}.\notag
\end{align}
To prove the second inequality, apply  H\"older's inequality to~\eqref{eq:L1} to 
obtain 
\[
\|J^\beta\phi(t)\|
\leq 
\int_0^t \omega_\beta (t-s)\|\phi(s)\|\, ds
\leq 
\left(\int_0^t \omega_\beta^2 (t-s)\, ds\right)^{1/2}\|\phi\|_{L^2(0,t;L^2)}
=
\frac{t^{\beta-1/2}}{\sqrt{2\beta-1}\,\Gamma(\beta)}\|\phi\|_{L^2(0,t;L^2)}
\]
for any $\beta>1/2$, which completes the proof of this lemma.
\end{proof}

\begin{lemma}\cite[Theorem~A.1]{McLean2012}\label{lem:rho}
For $t>0$,
\[
\int_0^t\langle\partial_s^{1-\alpha}v,v\rangle\,ds
	\ge\rho_\alpha t^{\alpha-1}\int_0^t\|v(s)\|^2\,ds
\quad\text{where}\quad
\rho_\alpha=\pi^{1-\alpha}\,\frac{(1-\alpha)^{1-\alpha}}{(2-\alpha)^{2-\alpha}}
\,\sin(\tfrac12\pi\alpha).
\]
\end{lemma}

The following estimate involving the force~$\textbf{F}$ is used several times 
in our analysis.

\begin{lemma}\label{lem:Fphi} 
If $\phi:[0,T]\to H^1(\Omega)$, then
\[
\bigl\|\nabla\cdot\bigl(\textbf{F}(t)\phi(t)\bigr)\bigr\|
\le\|\textbf{F}\|_{1,\infty}\|\phi(t)\|_{H^1(\Omega)}
\quad\text{for $0\le t\le T$.}
\]
\end{lemma}
\begin{proof}
The vector field identity $\nabla\cdot(\textbf{F}\phi)
=(\nabla\cdot\textbf{F})\phi+\textbf{F}\cdot(\nabla\phi)$ implies that
\begin{align*}
\bigl\|\nabla\cdot\bigl(\textbf{F}(t)\phi(t)\bigr)\bigr\|^2
    &\le\Bigl(\|\nabla\cdot\textbf{F}(t)\|_{L^\infty(\Omega)}^2
    +\|\textbf{F}(t)\|_{L^\infty(\Omega,\mathbb{R}^d)}^2\Bigr)
    \Bigl(\|\phi(t)\|^2+\|\nabla\phi(t)\|^2\Bigr)\\
    &\le\Bigl(\|\nabla\cdot\textbf{F}(t)\|_{L^\infty(\Omega)}
    +\|\textbf{F}(t)\|_{L^\infty(\Omega,\mathbb{R}^d)}\Bigr)^2
    \|\phi(t)\|_{H^1(\Omega)}^2 \\
    &\le\Bigl(\|\textbf{F}\|_{1,\infty} \|\phi(t)\|_{H^1(\Omega)}\Bigr)^2,
\end{align*}
which gives the desired estimate.
\end{proof}

We now recall a fundamental compactness result that will be used several times in the proofs of our main results.

\begin{lemma}[Aubin--Lions--Simon]\label{lem: Aubin}
Let $B_0\subset B_1\subset B_2$ be three Banach spaces. Assume that the 
embedding of $B_1$ in $B_2$ is continuous and that the embedding of $B_0$ in 
$B_1$ is compact. Let $p$~and $r$ satisfy $1\leq p,r\leq +\infty$. For $T>0$, 
define the Banach space
\[
E_{p,r}:=\bigl\{v\in L^p((0,T);B_0): \partial_t v\in L^r((0,T);B_2) \bigr\}
\]
with norm
\[
\|v\|_{E_{p,r}}:=\|v\|_{L^p((0,T);B_0)}+\|v'\|_{L^r((0,T);B_2)}.
\]
Then,
\begin{itemize}
\item the embedding $E_{p,r}\subset L^p((0,T),B_1)$ is compact 
when~$p<+\infty$, and
\item the embedding $E_{p,r}\subset C([0,T],B_1)$ is compact when
$p=+\infty$~and $r>1$. 
\end{itemize}
\end{lemma}
\begin{proof}
See, e.g.,~\cite[Theorem II.5.16]{BoyerFabrie2013}.
\end{proof}

\section{Galerkin approximation of the solution}\label{sec:Galerkin}
In this section we  prove existence and uniqueness of a finite-dimensional Galerkin approximation of the solution of~\eqref{prob}. This is a standard classical tool for deriving existence and regularity results for parabolic initial-boundary value problems; see, e.g., \cite[Section 7.1.2]{Evans10}.

Let $\{w_k\}_{k=1}^\infty$ be a complete set of eigenfunctions for the 
operator~$-\Delta$ in~$H_0^1(\Omega)$, with $\{w_k\}$ an orthonormal basis 
of~$L^2(\Omega)$ and an orthogonal basis of~$H_0^1(\Omega)$; 
see~\cite[Section 6.5.1]{Evans10}. For each positive integer~$m$, set 
$W_m=\operatorname{span}\{w_1,w_2,\ldots,w_m\}$ and consider $u_m:[0,T]\to  
W_m$ given by
\[
u_m(t):= \sum_{k=1}^m d_m^k(t) w_k(x).
\]
Let $\Pi_m$ be the orthogonal projector from~$L^2(\Omega)$ onto~$W_m$ defined by: for each $v\in L^2(\Omega)$, one has
\[
\Pi_m v\in W_m\ \text{ and }\ \langle\Pi_m v,w\rangle = \langle v,w\rangle\quad\text{for all $w\in W_m$.}
\]
The projections of the source term and initial data are denoted by
\[
g_m(t) := \Pi_mg(t)\ \text{ and} \  u_{0m} := \Pi_mu_0.
\]
We aim to choose the functions $d_m^k$ so that 
for $k=1$, $2$, \dots, $m$ and $t\in (0,T]$ one has
\begin{subequations} \label{dmk2}
\begin{align}
u_m' -\kappa_\alpha\partial_t^{1-\alpha}\Delta u_m
	+\Pi_m\bigl(\nabla\cdot(\textbf{F}(t)\partial_t^{1-\alpha}u_m)\bigr)
    &=g_m(t) \label{dmka2}
\intertext{and}
d_m^k(0) &= \langle u_0, w_k\rangle.  \label{dmkb2}
\end{align}
\end{subequations}

Existence and uniqueness of a solution to~\eqref{dmk2} are guaranteed by the following lemma.

\begin{lemma}\cite[Theorem 3.1]{LeMM16}\label{lem:existdmk}
Let $F\in W^{1,\infty}(0,T; L^\infty(\Omega))$~and $g\in L^1(0, T; L^2(\Omega))$. 
Then for each positive integer $m$, the system of equations \eqref{dmk2} has a solution $\{d_m^k\}_{k=1}^m$ with $u_m : [0,T] \to H^2(\Omega) \cap H_0^1(\Omega)$ absolutely continuous. This solution is unique among the space of absolutely continuous functions mapping $[0,T]$ to $H_0^1(\Omega)$.
\end{lemma}
\begin{proof}
Our argument is based mainly on the proof of~\cite[Theorem 3.1]{LeMM16}, 
but we fill a gap in that argument by verifying that $u_m$ is absolutely 
continuous. Define the linear operator~$B_m(t): W_m \to W_m$ by 
\[
\langle B_m(t)v, w \rangle:=-\kappa_\alpha \langle \Delta v, w\rangle 
+\bigl\langle\Pi_m\bigl(\nabla\cdot(\textbf{F}(t,\cdot)v)\bigr), 
    w\bigr\rangle 
\quad\text{for all $v$, $w\in W_m$,}
\]
and rewrite \eqref{dmka2} as
\[
u_m'(t)+B_m(t)\partial_t^{1-\alpha}u_m(t)=g_m(t).
\]
Formally integrating this equation in time we obtain the Volterra integral 
equation~\cite[p.1768]{LeMM16}:
\begin{equation}\label{umVIE}
u_m(t) + \int_{s=0}^t K_m(t,s)u_m(s)\,ds = G_m(t) \text{ for } 0\le t\le T,
\end{equation}
where 
\[
K_m(t,s)=B_m(t)\omega_\alpha(t-s)
    -\int_s^t B_m'(\tau)\omega_\alpha(\tau-s)\,d\tau 
\quad\text{and}\quad
G_m(t):=u_{0m}+\int_0^t g_m(s)\,ds.
\]
It is shown in \cite{LeMM16} that \eqref{umVIE} has a unique 
solution~$u_m\in C\bigl([0,T]; H_0^1(\Omega)\bigr)$.

Now, $g\in L^1\bigl(0, T; L^2(\Omega)\bigr)$ implies that 
$g_m\in L^1\bigl(0, T; L^2(\Omega)\bigr)$, and it follows that 
$G_m:[0,T]\to L^2(\Omega)$ is absolutely continuous. Furthermore, Theorem~2.5 
of~\cite{Diet10}  implies  (using the continuity of $u_m$) that $t \mapsto 
\int_{s=0}^t K_m(t,s)u_m(s)\,ds$ is absolutely continuous. Hence, \eqref{umVIE} 
shows that $u_m: [0,T]\to L^2(\Omega)$ is absolutely continuous.

We are now able to differentiate \eqref{umVIE} (to differentiate the integral 
term, imitate the calculation in the proof of \cite[Lemma 2.12]{Diet10}), 
obtaining
\[
u_m'(t)+\int_{s=0}^t B_m(t)\omega_\alpha(t-s)u_m'(s)\,ds=g_m(t)
    \quad\text{for almost all $t\in [0,T]$.}
\]
The absolute continuity of $u_m(t)$ implies that $\partial_t^{1-\alpha}u_m(t)$ 
exists for almost all $t\in [0,T]$ by~\cite[Lemma 2.12]{Diet10}. Hence from the above equation, 
$u_m$ satisfies~\eqref{dmka2}. From~\eqref{umVIE}, one 
sees immediately that $u_m$ satisfies~\eqref{dmkb2}, so we have demonstrated the 
existence of a solution to~\eqref{dmk2}.

To see that this solution of~\eqref{dmk2} is unique among the space of 
absolutely continuous functions, one can use the proof 
of~\cite[Theorem 3.1]{LeMM16} since the absolute continuity of the solution is 
now known \emph{a priori}.
\end{proof}

\section{Existence and uniqueness of the mild solution}\label{sec:mild}

In this section, we assume that $\alpha\in(0,1)$, $\textbf{F}\in W^{1,\infty}((0,T)\times\Omega)$ and that the initial data 
$u_0\in L^2(\Omega)$.

\subsection{\emph{A priori} estimates}

In order to prove \emph{a priori} estimates, we consider the integrated form of 
equation~\eqref{dmka2}:
\begin{align}
u_m(t) -\kappa_\alpha J^{\alpha}\Delta u_m (t)+\int_0^t\Pi_m\bigl(
	\nabla\cdot(\textbf{F}(s)\partial_t^{1-\alpha}u_m(s))\bigr)\,ds
    &=G_m(t), \label{int dmka2}
\end{align}
where $G_m(t)=\Pi_mG(t)$ as in~\eqref{umVIE}.

Let 
$C_{\mathrm{P}}$ denote the Poincar\'e constant for~$\Omega$, viz.,
$\|v\|^2\le C_{\mathrm{P}}\|\nabla v\|^2$ for $v\in H^1_0(\Omega)$.

\begin{lemma}\label{lem: priori weak1}
Let $m$ be a positive integer. Let $u_m(t)$ be the absolutely continuous 
solution of~\eqref{dmka2} that is guaranteed by Lemma~\ref{lem:existdmk}. Then for any $t\in [0,T]$ one has
\begin{align}
\cos(\alpha\pi/2)\int_0^t\|J^{\alpha/2}u_m(s)\|^2\,ds
 + \kappa_\alpha \int_0^t \|J^\alpha u_m (s)\|^2_{H^1(\Omega)}ds 
 &\leq C_1\int_0^t\|G_m(s)\|^2\,ds  \label{eq: pri weak0}
 \intertext{and}
 \int_0^t \|u_m(s)\|^2ds &\leq C_3\int_0^t\|G_m(s)\|^2\,ds,  \label{eq: pri weak00}
\end{align}
where 
\begin{align*}
C_1 &:= \frac{1+C_{\mathrm{P}}}{2}\left[1
+\frac{C_2\omega_{1+\alpha/2}^2(t)}{\cos(\alpha\pi/2)}\,
       E_{\alpha/2}\left(\frac{C_2\omega_{1+\alpha/2}(t)}{\cos(\alpha\pi/2) }t^\alpha\right)  \right],
\\
C_2 &:=2\biggl(1+\frac{\|\mathbf{F}\|_{\infty}^2}{\kappa_\alpha}\biggr) 
    +\frac{T^2\|\mathbf{F}'\|_{\infty}^2}{\kappa_\alpha}\,, \\
C_3 &:=2+\frac{C_1}{\kappa_\alpha}\left(4\|\mathbf{F}\|_{1,\infty}^2+2T^2\|\mathbf{F}'\|_{1,\infty}^2\right) .
\end{align*}
\end{lemma}
\begin{proof}
Taking the inner product of both sides of~\eqref{int dmka2} 
with~$J^{\alpha}u_m(t)\in W_m$ then integrating by parts with respect 
to~$x$, we obtain
\begin{align}\label{eq: pri weak1}
\langle J^{\alpha} u_m (t),  u_m(t) \rangle 
    + \kappa_{\alpha}\|J^{\alpha}\nabla u_m(t)\|^2 
    &=\left\langle \int_0^t\textbf{F}(s)\partial_t^{1-\alpha}u_m(s)ds, 
        J^{\alpha}\nabla u_m(t) \right\rangle  + \langle G_m(t),J^{\alpha}u_m(t)\rangle\nonumber\\
&\leq\frac{\kappa_\alpha}{2}\|J^{\alpha}\nabla u_m(t) \|^2 
+\frac{1}{2\kappa_{\alpha}}\left\|
    \int_0^t\textbf{F}(s)\partial_t^{1-\alpha} u_m(s)\,ds\right\|^2\nonumber\\
&\qquad +\frac14 \|G_m(t)\|^2 + \|J^\alpha u_m(t) \|^2. 
\end{align}
Integrating by parts with respect to the time variable, and using Minkowski's integral inequality and H\"older's inequality, we have
\begin{align}\label{eq: pri weak2}
\biggl\|\int_0^t\textbf{F}(s)\partial_t^{1-\alpha}u_m(s)\,ds\biggr\|^2
&=\biggl\|\textbf{F}(t) J^{\alpha}u_m(t)
    - \int_0^t\textbf{F}'(s) J^{\alpha}u_m(s)\,ds\biggr\|^2\nonumber\\
&\leq 2\|\textbf{F}\|_{\infty}^2\| J^{\alpha}u_m(t)\|^2
+
2\|\textbf{F}'\|_{\infty}^2 \left(\int_0^t \|J^{\alpha}u_m(s)\|\,ds\right)^2\nonumber\\
&\leq 
2\|\textbf{F}\|_{\infty}^2\| J^{\alpha}u_m(t)\|^2
+
2t\|\textbf{F}'\|_{\infty}^2\int_0^t \|J^{\alpha}u_m(s)\|^2\,ds.
\end{align}
It follows from~\eqref{eq: pri weak1} and~\eqref{eq: pri weak2} that
\begin{align*}
\langle J^{\alpha} u_m (t),  u_m(t)  \rangle
+ 
\frac{\kappa_{\alpha}}{2}\|J^{\alpha}\nabla u_m(t)\|^2 
&\leq \frac14 \|G_m(t)\|^2
+\left(1+\frac{\|\textbf{F}\|_{\infty}^2}{\kappa_\alpha}\right)
    \| J^{\alpha}u_m(t)\|^2\\
&\qquad+\frac{t\|\textbf{F}'\|_{\infty}^2}{\kappa_\alpha}
    \int_0^t \|J^{\alpha}u_m(s)\|^2\,ds.
\end{align*}
Integrating in time and invoking Lemma~\ref{lem:pd2}, we deduce that 
\begin{align}\label{eq: pri weak3}
\cos(\alpha\pi/2) & \int_0^t \|J^{\alpha/2}u_m(s)\|^2\,ds
+\kappa_{\alpha}\int_0^t \|J^{\alpha}\nabla u_m(s)\|^2\,ds
\leq \frac12 \int_0^t  \|G_m(s)\|^2 \,ds\nonumber\\
    &\qquad{}+2\left(1+\frac{\|\textbf{F}\|_{\infty}^2}{\kappa_\alpha}\right) 
    \int_0^t \| J^{\alpha}u_m(s)\|^2\,ds
+\frac{2\|\textbf{F}'\|_{\infty}^2}{\kappa_\alpha}
    \int_0^t s\int_0^s \|J^{\alpha}u_m(\tau)\|^2d\tau \,ds\nonumber\\
&\leq\frac12 \int_0^t  \|G_m(s)\|^2 \,ds
    + C_2 \int_0^t \| J^{\alpha}u_m(s)\|^2\,ds.
\end{align}
But Lemma~\ref{lem:minsk} gives us
\begin{equation}\label{eq: pd21}
\| J^{\alpha}u_m(s)\|^2 = \| J^{\alpha/2}(J^{\alpha/2}u_m)(s)\|^2
	\leq \omega_{1+\alpha/2}(s)J^{\alpha/2}(\| J^{\alpha/2}u_m\|^2)(s).
\end{equation}
Thus, setting $\psi_m(t):=J^{1}(\|J^{\alpha/2} u_m\|^2)(t)$, we deduce 
from~\eqref{eq: pri weak3} that
\begin{align*}
\psi_m(t) 
&\leq \frac{1}{2\cos(\alpha\pi/2) } \int_0^t  \|G_m(s)\|^2 \,ds 
	+ \frac{C_2\omega_{1+\alpha/2}(t)}{\cos(\alpha\pi/2) } 
	J^{1+\alpha/2}(\| J^{\alpha/2}u_m\|^2)(t)\\
&=
\frac{1}{2\cos(\alpha\pi/2) } \int_0^t  \|G_m(s)\|^2 \,ds 
	+ \frac{C_2\omega_{1+\alpha/2}(t)}{\cos(\alpha\pi/2) } J^{\alpha/2}\psi_m(t).
\end{align*}
Applying Lemma~\ref{lem:Gronwall}, one obtains
\begin{equation}\label{eq: pri weak4}
\psi_m(t) 
\leq E_{\alpha/2}\left(
	\frac{C_2\omega_{1+\alpha/2}(t)}{\cos(\alpha\pi/2)}\,t^\alpha\right)
	\,\frac{1}{2\cos(\alpha\pi/2) }  \int_0^t \|G_m(s)\|^2 \,ds 
    \quad\text{for $0\leq t\leq T$.}
\end{equation}
This inequality and~\eqref{eq: pd21} together yield
\begin{align*}
\int_0^t\|J^\alpha u_m(s)\|^2\,ds
&\leq \omega_{1+\alpha/2}(t)J^{\alpha/2}\psi_m(t)\\
    &\le    
    \frac{\omega_{1+\alpha/2}(t)}{2\cos(\alpha\pi/2)}\int_0^t\omega_{\alpha/2}(t-s)
     E_{\alpha/2}\left(\frac{C_2\omega_{1+\alpha/2}(s)}{\cos(\alpha\pi/2) }s^\alpha\right)
     	\int_0^s\|G_m(z)\|^2\,dz\,ds\nonumber\\
    &\le\frac{\omega_{1+\alpha/2}^2(t)}{2\cos(\alpha\pi/2)}
       E_{\alpha/2}\left(
	\frac{C_2\omega_{1+\alpha/2}(t)}{\cos(\alpha\pi/2)}\,t^\alpha\right)
	\int_0^t\|G_m(s)\|^2\,ds.
\end{align*}
Now \eqref{eq: pri weak0} follows immediately on recalling~\eqref{eq: pri weak3}--\eqref{eq: pri weak4} and the Poincar\'e inequality.

In a similar fashion, we take the inner product of both sides 
of~\eqref{int dmka2} with $u_m(t)\in W_m$ and then integrate by parts with respect to $x$, to obtain
\begin{multline}\label{eq: pri weak5}
\|u_m(t)\|^2+\kappa_{\alpha}\langle J^{\alpha}\nabla u_m(t),
    \nabla u_m(t)\rangle
= -\left\langle 
\int_0^t\nabla\cdot\bigl(\textbf{F}(s)\partial_t^{1-\alpha}u_m(s)\bigr)\,ds,  
    u_m(t)\right\rangle+\langle G_m,u_m\rangle\\
\leq \|G_m(t)\|^2 + \frac12 \|u_m(t)\|^2
+\left\|\int_0^t \nabla\cdot\bigl( 
    \textbf{F}(s)\partial_t^{1-\alpha}u_m(s)\bigr)\,ds\right\|^2.
\end{multline}
Using Lemma~\ref{lem:Fphi} and the same arguments as in the proof of~\eqref{eq: pri weak2}, we also have
\begin{equation}\label{eq: pri weak6}
\left\|\int_0^t \nabla\cdot\bigl( 
    \mathbf{F}(s)\partial_t^{1-\alpha}u_m(s)\bigr)\,ds\right\|^2
\leq2\|\mathbf{F}\|_{1,\infty}^2\| J^\alpha u_m(t)\|_{H^1(\Omega)}^2
+2t\|\mathbf{F}'\|_{1,\infty}^2
    \int_0^t \|J^\alpha u_m(s)\|_{H^1(\Omega)}^2\,ds. 
\end{equation}
This estimate and \eqref{eq: pri weak5} together imply
\begin{align*}
\frac12 \|u_m(t)\|^2 + \kappa_{\alpha}\langle J^{\alpha}\nabla u_m (t), \nabla u_m(t)\rangle 
    &\leq \|G_m(t)\|^2  + 2\|\textbf{F}\|_{1,\infty}^2\| J^{\alpha}u_m(t)\|_{H^1(\Omega)}^2  \\
    &\qquad +2t\|\textbf{F}'\|_{1,\infty}^2     \int_0^t \|J^{\alpha}u_m(s)\|_{H^1(\Omega)}^2\,ds.
\end{align*}
Integrating in time, we get
\[
\int_0^t \|u_m(s)\|^2\,ds\leq2\int_0^t \|G_m(s)\|^2\,ds
+\bigl(4\|\textbf{F}\|_{1,\infty}^2+ 
    2t^2\|\textbf{F}'\|_{1,\infty}^2\bigr)
    \int_0^t\|J^\alpha u_m(s)\|_{H^1(\Omega)}^2\,ds.
\]
Now apply the inequality~\eqref{eq: pri weak0} to complete the proof.
\end{proof}
\begin{lemma}\label{lem: priori weak2}
Let $m$ be a positive integer, and let $u_m(t)$ be the absolutely continuous 
solution of~\eqref{dmka2} that is guaranteed by Lemma~\ref{lem:existdmk}. Then,
for any $t\in [0,T]$,
\begin{align}
\cos(\alpha\pi/2)\int_0^t\|J^{\alpha/2} \nabla u_m(s)\|^2\,ds
 + \kappa_\alpha \int_0^t \|J^\alpha \Delta u_m (s)\|^2\,ds 
    &\leq C_4\int_0^t\|G_m(s)\|^2\,ds  \label{eq: pri weak000}  \\
\intertext{and}
\|J^1 u_m(t)\|_{H^1(\Omega)}^2 &\leq C_5\int_0^t\|G_m(s)\|^2\,ds,  \label{eq: pri weak000b}
\end{align}
where 
\[
C_4:=\frac{2}{\kappa_\alpha} 
 +\frac{2C_1}{\kappa_\alpha^2}\bigl(2\|\mathbf{F}\|_{1,\infty}^2 
 + T^2\|\mathbf{F}'\|_{1,\infty}^2\bigr)
\quad\text{and}\quad
C_5:=\frac{C_4T^{1-\alpha}(1+C_{\mathrm{P}})}%
{(1-\alpha)\cos(\alpha\pi/2)\Gamma(1-\alpha/2)^2}\,.
\]
\end{lemma}
\begin{proof}
Taking the inner product of both sides of~\eqref{int dmka2} 
with~$-J^{\alpha}\Delta u_m(t)\in W_m$ and then integrating by parts with 
respect to~$x$, we obtain
\begin{align*}
\langle J^{\alpha} \nabla u_m (t)&, \nabla u_m(t)  \rangle
    +\kappa_{\alpha}\|J^{\alpha}\Delta u_m(t)\|^2 \\
&= 
\left\langle\int_0^t\nabla\cdot(\textbf{F}(s)\partial_t^{1-\alpha}u_m(s))\,ds, 
    J^{\alpha}\Delta u_m(t)\right\rangle
-\langle G_m(t),J^{\alpha} \Delta u_m(t)\rangle\\
&\leq 
\frac{\kappa_{\alpha}}{2}\|J^{\alpha}\Delta u_m(t) \|^2 
    +\frac1{\kappa_{\alpha}}\|G_m(t)\|^2+\frac{1}{\kappa_{\alpha}}
    \left\|\int_0^t\nabla\cdot(\textbf{F}(s)\partial_t^{1-\alpha}u_m(s))\,
        ds\right\|^2.
\end{align*}
This inequality and~\eqref{eq: pri weak6} together imply
\begin{align*}
\langle J^{\alpha} \nabla u_m (t), \nabla u_m(t)  \rangle
+ \frac{\kappa_{\alpha}}{2}\|J^{\alpha}\Delta u_m(t)\|^2 
&\leq 
\frac1{\kappa_{\alpha}}\|G_m(t)\|^2
+\frac{2}{\kappa_{\alpha}}\|\textbf{F}\|_{1,\infty}^2
    \| J^{\alpha}u_m(t)\|_{H^1(\Omega)}^2\\
&\qquad{}+\frac{2t}{\kappa_{\alpha}}\,\|\textbf{F}'\|_{1,\infty}^2
    \int_0^t \|J^{\alpha}u_m(s)\|_{H^1(\Omega)}^2\,ds.
\end{align*}
Integrating in time and invoking Lemma~\ref{lem:pd2}, we deduce that 
\begin{multline*}
2 \cos(\alpha\pi/2) J^{1}(\|J^{\alpha/2} \nabla u_m\|^2)(t)
+ \kappa_\alpha \int_0^t \|J^\alpha \Delta u_m (s)\|^2\,ds\\
\leq\frac{2}{\kappa_{\alpha}}\int_0^t \|G_m(s)\|^2\,ds
 + \frac{2}{\kappa_{\alpha}}
    (2\|\textbf{F}\|_{1,\infty}^2 + t^2\|\textbf{F}'\|_{1,\infty}^2)
    \int_0^t\|J^{\alpha}u_m(s)\|_{H^1(\Omega)}^2\,ds,
\end{multline*}
which, after applying inequality~\eqref{eq: pri weak0} of 
Lemma~\ref{lem: priori weak1}, completes the proof of~\eqref{eq: pri weak000}.

Applying~\eqref{eq:L1} with $\phi= J^{\alpha/2} u_m$ and $\beta=1-\alpha/2$ 
gives
\begin{align*}\label{eq:pd22}
 \|J^1 u_m(t) \|_{H^1(\Omega)}
&=\|J^{1-\alpha/2} J^{\alpha/2} u_m(t)\|_{H^1(\Omega)}
\leq 
  J^{1-\alpha/2}(\|J^{\alpha/2} u_m\|_{H^1(\Omega)})(t)
=(\omega_{1-\alpha/2}*z)(t),
\end{align*}
where $z(t)=\|J^{\alpha/2} u_m(t)\|_{H^1(\Omega)}$.
Using Young's convolution inequality we get
\begin{align*}
\|J^1 u_m(t)\|_{H^1(\Omega)}^2&\le\|\omega_{1-\alpha/2}*z\|_{L^\infty(0,t)}^2
\leq\|\omega_{1-\alpha/2}\|_{L^2(0,t)}^2\|z\|_{L^2(0,t)}^2\\
&= \frac{t^{1-\alpha}}{(1-\alpha)\Gamma(1-\alpha/2)^2}\, 
\int_0^t\|J^{\alpha/2} u_m(s)\|_{H^1(\Omega)}^2\,ds.
\end{align*}
The inequality \eqref{eq: pri weak000b} now follows immediately from~\eqref{eq: 
pri weak000}. 
\end{proof}

\subsection{The mild solution}
Our assumption that $\Omega$ has a $C^2$ boundary ensures that if $v\in 
H^1_0(\Omega)$ satisfies $\Delta v\in L^2(\Omega)$, then $v\in H^2(\Omega)$.
Moreover, there is a regularity constant~$C_{\mathrm{R}}$, depending only 
on~$\Omega$, such that
\begin{equation}\label{CR}
\|v\|_{H^2(\Omega)}\le C_{\mathrm{R}}\|\Delta v\|  \ \text{ for }v\in H^1_0(\Omega).
\end{equation}

Our next result requires a strengthening of the regularity hypothesis on \textbf{F}. 

\begin{theorem}\label{the:exit0}
Assume that $u_0\in L^2(\Omega), \textbf{F} \in W^{2,\infty}((0,T)\times\Omega)$ and 
$g\in L^2\bigl(0,T;L^2(\Omega)\bigr)$. Then there exists a unique mild solution 
$u$ of~\eqref{prob} (in the sense of Definition~\ref{def:mild}) such that
\begin{align}\label{eq:ubound0}
\|u\|_{L^2(0,T;L^2)}^2+\|J^\alpha u\|_{L^2(0,T;H^2)}^2\leq 
\bigl(C_4+\kappa_\alpha^{-1}C_4C_{\textrm{R}}\bigr) \|G\|_{L^2}^2.
\end{align}
\end{theorem}
\begin{proof}
In order to prove the existence of a mild solution, we first prove the convergence of the approximate solutions~$u_m$, and then find the limit of equation~\eqref{int dmka2} as~$m$ tends to infinity.

Note first that $\|G_m(s)\|\le\|G(s)\|$ because
\[
|\langle G_m(s),w\rangle|=|\langle G(s),\Pi_mw\rangle|
    \le\|G(s)\|\|\Pi_mw\|\le\|G(s)\|\|w\|
    \quad\text{for all $w\in L_2(\Omega)$.}
\]
Hence Lemma~\ref{lem: priori weak1} shows that the sequence 
$\{\partial_t(J^1u_m)\}_{m=1}^\infty = \{u_m\}_{m=1}^\infty$ is 
bounded in $L^2\bigl(0,T;L^2(\Omega)\bigr)$, and Lemma~\ref{lem: priori weak2} 
shows that the sequence~$\{J^1u_m\}_{m=1}^\infty$ is bounded 
in~$L^\infty\bigl(0,T;H^1_0(\Omega)\bigr)$.  Applying Lemma~\ref{lem: 
Aubin} with $B_0=H^1_0(\Omega)$, $B_1=B_2=L^2(\Omega)$, $p=+\infty$~and $r=2$,
it follows that there exists a subsequence of~$\{J^1u_m\}_{m=1}^\infty$, again 
denoted by~$\{J^1u_m\}_{m=1}^\infty$, and a 
$v\in C\bigl([0,T];L^2(\Omega)\bigr)$, such that 
\begin{equation}\label{con:1 weak}
\text{$J^1u_m\rightarrow v$ strongly in $C\bigl([0,T];L^2(\Omega)\bigr)$.}
\end{equation}
Furthermore, from the above bounds on $\{J^1u_m\}_{m=1}^\infty$ and 
well-known results~\cite[Theorem II.2.7]{BoyerFabrie2013} for weak and 
weak-$\star$ compactness, by choosing sub-subsequences we get
\begin{equation}\label{con:3 weak}
J^1u_m \to v \text{ weak-$\star$ in } L^\infty\bigl(0,T;H^1_0(\Omega)\bigr) \text{ and }
u_m=\partial_t(J^1u_m)\to\partial_t v \text{ weakly in }L^2\bigl(0,T;L^2(\Omega)\bigr).
\end{equation}

By letting $u:=\partial_t v\in L^2\bigl(0,T;L^2(\Omega)\bigr)$, we have 
$v=J^1u$. It remains to prove that $J^\alpha u_m$ converges weakly 
to~$J^\alpha u$ in~$L^2\bigl(0,T;H^2(\Omega)\bigr)$. Applying 
Lemma~\ref{lem:L1} with $\phi=J^1u_m$~and $\beta=\alpha$, for any $t\in[0,T]$ we deduce that
\[
\|J^{1+\alpha} u_m(t) \|_{H^1(\Omega)}
\leq 
\frac{t^\alpha}{\Gamma(\alpha+1)} \|J^1u_m\|_{L^{\infty}(0,t;H^1(\Omega))}.
\]
This inequality, together with Lemma~\ref{lem: priori weak2}, implies that the 
sequence~$\{J^{1+\alpha}u_m\}_{m=1}^\infty$ is bounded in~$L^\infty\bigl(0,T;H^1_0(\Omega)\bigr)$. Also,
Lemma~\ref{lem: priori weak1} shows that the sequence~$\{\partial_t(J^{1+\alpha} u_m)\}_{m=1}^\infty
=\{J^\alpha u_m\}_{m=1}^\infty$ is bounded in~$L^2\bigl(0,T; H^1_0(\Omega)\bigr)$. 
It now follows from Lemma~\ref{lem: Aubin}, again with $B_0=H^1_0(\Omega)$,
$B_1=B_2=L^2(\Omega)$, $p=+\infty$~and $r=2$, that there exists a subsequence 
of  $\{J^{1+\alpha}u_m\}_{m=1}^\infty$ (still denoted 
by~$\{J^{1+\alpha}u_m\}_{m=1}^\infty$) and 
$\bar{u}\in C\bigl([0,T];L^2(\Omega)\bigr)$ such 
that 
\begin{equation}\label{con:2 weak}
\text{$J^{1+\alpha} u_m\to\bar u$ strongly in 
$C\bigl([0,T];L^2(\Omega)\bigr)$.}
\end{equation}
Furthermore, from the upper bound~\eqref{eq: pri weak000} 
of~$\{\partial_t(J^{1+\alpha}u_m)\}_{m=1}^\infty$ 
in~$L^2\bigl(0,T;H^2(\Omega)\bigr)$, by choosing a subsequence one gets
\begin{align}\label{con:4 weak}
\text{$J^\alpha u_m = \partial_t(J^{1+\alpha}u_m)\to\partial_t \bar u$
weakly in $L^2\bigl(0,T;H^2(\Omega)\bigr)$.}
\end{align}
On the other hand, by applying Lemma~\ref{lem:L1} with 
$\phi=J^1(u_m-u)$ and $\beta=\alpha$, we deduce that for any $t\in[0,T]$ one has
\[
\|J^{1+\alpha} (u_m -u)(t)\|_{L^2(\Omega)}\leq 
\frac{t^\alpha}{\Gamma(\alpha+1)} \|J^1(u_m-u)\|_{L^{\infty}(0,t;L^2(\Omega))}.
\]
Hence, \eqref{con:1 weak} implies that 
$\lim_{m\to\infty}\|J^{1+\alpha}(u_m -u)\|_{L^\infty(0,T;L^2(\Omega))}=0$. 
Recalling~\eqref{con:2 weak}, we have $\bar{u} =J^{1+\alpha} u$.
By choosing subsequences, we obtain
\begin{align}
J^{1+\alpha} u_m&\to  J^{1+\alpha} u\,
\text{ strongly in $C\bigl([0,T];L^2(\Omega)\bigr)$,}\label{con:6 weak}\\
J^{1+\alpha} u_m&\to  J^{1+\alpha} u\,
\text{ weak-$\star$ in $L^\infty\bigl(0,T;H^1_0(\Omega)\bigr)$,}
\label{con:7 weak}\\
J^{\alpha} u_m&\to  J^{\alpha} u\,
\text{ weakly in $L^2\bigl(0,T;H^2(\Omega)\bigr)$,} \label{con:5 weak}
\end{align}
where we used the boundedness of $\{J^{1+\alpha} u_m\}_{m=1}^\infty$ 
in~$L^\infty\bigl(0,T;H^1_0(\Omega)\bigr)$ that was already mentioned, and 
\eqref{con:4 weak}. 

Multiplying both sides of~\eqref{int dmka2} by a test function $\xi\in 
C_c^\infty((0,T)\times \Omega)$, integrating over $(0,T)\times\Omega$
and noting that $\Pi_m$ is a self-adjoint operator on~$L^2(\Omega)$ gives
\begin{align}\label{lim weak1}
\langle u_m,\xi \rangle_{L^2(0,T;L^2)}
-\kappa_\alpha\langle J^{\alpha}\Delta u_m ,\xi \rangle_{L^2(0,T;L^2)}
	+\langle h_m,\Pi_m\xi \rangle_{L^2(0,T;L^2)}
	= \langle G_m,\xi \rangle_{L^2(0,T;L^2)},
\end{align}
where $h_m(t):=\int_0^t \nabla\cdot(\textbf{F}(s)\partial_t^{1-\alpha}u_m(s))\,ds$.
Using~\eqref{con:3 weak}~and \eqref{con:5 weak}, as~$m\to\infty$ one has
\begin{equation}\label{lim weak2}
\begin{aligned}
\langle u_m,\xi \rangle_{L^2(0,T;L^2)}
    &\to \langle u,\xi \rangle_{L^2(0,T;L^2)},\\
\langle G_m,\xi \rangle_{L^2(0,T;L^2)}
    &\to \langle G,\xi \rangle_{L^2(0,T;L^2)},\\
\langle J^{\alpha}\Delta u_m ,\xi \rangle_{L^2(0,T;L^2)} 
    &\to \langle J^{\alpha}\Delta u ,\xi \rangle_{L^2(0,T;L^2)}.
\end{aligned}
\end{equation}
To find the limit of the most complicated term~$\langle h_m,\Pi_m\xi\rangle_{L^2(0,T;L^2)}$ in~\eqref{lim weak1}, we first integrate by parts 
twice with respect to the time variable:
\begin{equation}\label{lim weak3}
\begin{aligned}
h_m(t)=\int_0^t \nabla\cdot(\textbf{F}(s)\partial_t^{1-\alpha}u_m(s))\,ds
    &=\nabla\cdot(\textbf{F}(t)J^{\alpha}u_m(t))
    -\int_0^t \nabla\cdot(\textbf{F}'(s)J^{\alpha}u_m(s))\,ds\\
&=\nabla\cdot(\textbf{F}(t)J^{\alpha}u_m(t))
    -\nabla\cdot(\textbf{F}'(t)J^{1+\alpha}u_m(t))\\
&\qquad{}+\int_0^t \nabla\cdot(\textbf{F}''(s)J^{1+\alpha}u_m(s))\,ds.
\end{aligned}
\end{equation} 
It now follows from the boundedness of $\{J^{\alpha}u_m\}_{m=1}^\infty$ and 
$\{J^{1+\alpha}u_m\}_{m=1}^\infty$ in $L^2\bigl(0,T;H^2(\Omega)\bigr)$ and 
$L^\infty\bigl(0,T;H^1_0(\Omega)\bigr)$, respectively, that 
$\{h_m\}_{m=1}^\infty$ is bounded in~$L^2\bigl(0,T;L^2(\Omega)\bigr)$. 
Hence,
\begin{equation}\label{lim:weak4}
\lim_{m\rightarrow\infty} \langle h_m,\Pi_m\xi - \xi \rangle_{L^2(0,T;L^2)} = 0.
\end{equation}
On the other hand, by using~\eqref{lim weak3} and integration by parts with 
respect to~$x$, we have
\begin{align*}
\langle h_m,\xi \rangle_{L^2(0,T;L^2)}
&=\langle \nabla\cdot(\textbf{F}J^{\alpha}u_m),\xi \rangle_{L^2(0,T;L^2)}
    -\langle\nabla\cdot \left(\textbf{F}'J^{1+\alpha}u_m \right),
            \xi \rangle_{L^2(0,T;L^2)}\\
&\qquad{}-\int_0^T\int_\Omega \biggl(
    \int_0^t \textbf{F}''(s)J^{1+\alpha}u_m(s)\,ds
    \biggr)\cdot\nabla\xi(t)dx\,dt.
\end{align*}
Combining this identity with~\eqref{con:6 weak}--\eqref{con:5 weak} gives
\begin{align*}
\lim_{m\rightarrow\infty} \langle h_m,\xi \rangle_{L^2(0,T;L^2)}
	&=\langle \nabla\cdot(\textbf{F}J^{\alpha}u),\xi \rangle_{L^2(0,T;L^2)}
		-\langle\nabla\cdot(\textbf{F}'J^{1+\alpha}u),\xi \rangle_{L^2(0,T;L^2)}\\
	&\qquad-\int_0^T\int_\Omega \biggl(
    \int_0^t \textbf{F}''(s)J^{1+\alpha}u(s)\,ds\biggr)\cdot\nabla\xi(t)dx\,dt\\
&=\langle \nabla\cdot(\textbf{F}J^{\alpha}u),\xi \rangle_{L^2(0,T;L^2)}
- \int_0^T\int_\Omega \nabla\cdot\biggl(
    \int_0^t \textbf{F}'(s)J^{\alpha}u(s)\,ds\biggr)\,\xi(t)\,dx\,dt.
\end{align*}
Now invoking~\eqref{lim:weak4} yields
\begin{equation}
\lim_{m\rightarrow\infty} \langle h_m,\Pi_m\xi \rangle_{L^2(0,T;L^2)}
	=\langle \nabla\cdot(\textbf{F}J^{\alpha}u),\xi \rangle_{L^2(0,T;L^2)}  
	  -\int_0^T\int_\Omega \nabla\cdot\left(\int_0^t\textbf{F}'(s)J^{\alpha}u(s)\,ds\right)\,\xi(t)\,dx\,dt. \label{lim:weak5}
\end{equation}

Let $m\rightarrow\infty$ in~\eqref{lim weak1}. Using \eqref{lim weak2}~and \eqref{lim:weak5}, we deduce that for 
any~$\xi\in C_c^\infty((0,T)\times\Omega)$ one has
\begin{align*}
\biggl\langle u-\kappa_\alpha J^{\alpha}\Delta u 
	+ \nabla\cdot(\textbf{F}J^{\alpha}u)
-\nabla\cdot\biggl(\int_0^t \textbf{F}'(s)J^{\alpha}u(s)\,ds\biggr),
    \xi\biggr\rangle_{L^2(0,T;L^2)}
	= \langle g,\xi \rangle_{L^2(0,T;L^2)}.
\end{align*}
Since $C_c^\infty((0,T)\times\Omega)$ is dense in $L^2((0,T)\times\Omega)$, the above equation also holds true for any test function $\xi\in L^2((0,T)\times\Omega)$. Hence, $u$ satisfies~\eqref{prob0} a.e. on $(0,T)\times\Omega$.

The weak convergence of $u_m$ described in \eqref{con:3 weak},  
and~\cite[Corollary II.2.8]{BoyerFabrie2013} with \eqref{eq: pri weak00} 
together yield $\|u\|_{L^2(0,T;L^2)}^2\le C_3\|G\|_{L^2(0,T;L^2)}^2$. 
Similarly,  \eqref{con:5 weak} and \eqref{eq: pri weak000} imply that
$\kappa_\alpha\|J^\alpha\Delta u\|_{L^2(0,T;H^2)}^2
\leq C_4\|G\|_{L^2(0,T;L^2)}^2$. Thus, \eqref{eq:ubound0} is proved.

The uniqueness of the solution~$u$ follows from linearity~and \eqref{eq:ubound0},
because if $u_0=0$~and $g=0$, then $G=0$ and hence $u=0$.
\end{proof}

\section{Existence and uniqueness of the classical solution}
\label{sec:classical}
\begin{assumption}\label{ass:12}
In the rest of this paper, we assume that  
\[
\frac12<\alpha<1. 
\]
\end{assumption}

Assumption~\ref{ass:12} is not overly restrictive because \eqref{prob} is 
usually considered as a variant of the case~$\alpha=1$.  
We cannot avoid this restriction on $\alpha$ in Sections~\ref{sec:classical} and~\ref{sec:regclassical} since our analysis makes heavy use 
of $\partial_t^{1-\alpha}u$, and for typical solutions~$u$ of~\eqref{prob}, it 
will turn out that $\|\partial_t^{1-\alpha}u\|_{L^2(0,T;L^2(\Omega))} < \infty$ 
only for $1/2<\alpha<1$. To see this heuristically, assume that 
$u(x,t)=\phi(x)+v(x,t)$, where $v$ vanishes as~$t\to 0$ so $\phi(x)$ is the 
dominant component near~$t=0$; then $\partial_t^{1-\alpha}u(x,t) = 
\phi(x)\omega_\alpha(t)+\partial_t^{1-\alpha}v(x,t) \approx 
\phi(x)\omega_\alpha(t)$ near~$t=0$, and $\int_0^T \omega_\alpha^2(t)\,dt$ is 
finite only if $\alpha> 1/2$.

\subsection{\emph{A priori} estimates}
Since $\partial_t^{1-\alpha}1 = \omega_\alpha(t)$, we can rewrite~\eqref{dmka2} in terms of 
\[
v_m(t):= u_m(t)-u_{0m}
\] 
as
\begin{equation}\label{dmka3}
v_m'-\kappa_\alpha\partial_t^{1-\alpha}\Delta v_m
	+\Pi_m\bigl(\nabla\cdot(\textbf{F}(t)\partial_t^{1-\alpha}v_m)\bigr)\\
	= \Pi_m g(t)-\omega_\alpha(t)\bigl[ 
    \Pi_m\bigl(\nabla\cdot(\textbf{F}(t)u_{0m}\bigr)
	-\kappa_\alpha\Delta u_{0m}\bigr].
\end{equation}
We will require the following bound for~$u_{0m}$.

\begin{lemma}\label{lem:u0m}
If $u_0\in H^2(\Omega)$, then 
$\|u_{0m}\|_{H^2(\Omega)}\le C_{\mathrm{R}}\|u_0\|_{H^2(\Omega)}$ for all $m$, where the constant~$C_R$ was defined in~\eqref{CR}.
\end{lemma}
\begin{proof}
Write $d^k_{0m}=d^k_m(0)=\langle u_0,w_k\rangle$ for $1\le k\le m$, and let 
$\lambda_k>0$ denote the $k$th Dirichlet eigenvalue of the Laplacian so that 
$-\Delta w_k=\lambda_kw_k$ for all~$k$.  In this way,
\[
u_{0m}(x)=\sum_{k=1}^md^k_{0m}w_k(x)
\quad\text{and}\quad
-\Delta u_{0m}(x)=\sum_{k=1}^md^k_{0m}\lambda_kw_k(x).
\]
If $u_0\in H^2(\Omega)$ then $\Delta u_0\in L_2(\Omega)$ so, using Parseval's 
identity,
\[
\|\Delta u_{0m}\|^2=\sum_{k=1}^m|\langle\Delta u_{0m},w_k\rangle|^2
	=\sum_{k=1}^m\lambda_k^2|\langle u_0,w_k\rangle|^2
	\le\sum_{k=1}^\infty\lambda_k^2|\langle u_0,w_k\rangle|^2
	=\sum_{k=1}^\infty |\langle u_0, \Delta w_k\rangle|^2 
	=\|\Delta u_0\|^2.
\]
Thus, $\|u_{0m}\|_{H^2(\Omega)}\le C_{\mathrm{R}}\|\Delta u_{0m}\|
\le C_{\mathrm{R}}\|\Delta u_0\|\le C_{\mathrm{R}}\|u_0\|_{H^2(\Omega)}$.
\end{proof}

We now prove upper bounds for $\|\partial_t^{1-\alpha}v_m(t)\|$ and $J^\alpha\bigl(\|\partial_t^{1-\alpha}\nabla v_m\|^2\bigr)(t)$ for any $t\in[0,T]$.  The argument used in the following lemma is based on the proof 
of~\cite[Theorem 3.1]{HLeS17}.

\begin{lemma}\label{thm:energyum1}
Let $m$ be a positive integer. Let $v_m(t)$ be the absolutely continuous 
solution of~\eqref{dmka3} that is guaranteed by Lemma~\ref{lem:existdmk}. Then,
for almost all $t\in [0,T]$,
\begin{equation}
\|\partial_t^{1-\alpha}v_m(t)\|^2 
	\le E_{2\alpha-1}\bigl(C_7t^{2\alpha-1}\bigr) 
	\biggl(C_6\|u_{0}\|_{H^2(\Omega)}^2+\int_0^t\|g(s)\|^2\,ds\biggr)
    \label{thmenergy1}
\end{equation}
and  
\begin{align}
J^\alpha\bigl(\|\partial_t^{1-\alpha}\nabla v_m\|^2\bigr)(t) 
\le\frac{1+C_7E_{2\alpha-1}\bigl(C_7t^{2\alpha-1}\bigr) 
\omega_{2\alpha}(t)}{\kappa_\alpha}
	\biggl(C_6\|u_0\|_{H^2(\Omega)}^2+\int_0^t\|g(s)\|^2\,ds\biggr),
\label{thmenergy2}
\end{align}
where 
\[
C_6:= C_{\mathrm{R}}^2 \bigl(\kappa_\alpha+\|\mathbf{F}\|_{1,\infty}\bigr)^2
\quad\text{and}\quad
C_7:=\frac{\Gamma(2\alpha-1)}{\Gamma(\alpha)^2}\biggl(1
	+\frac{T^{2\alpha-1}}{(2\alpha-1)\Gamma(\alpha)^2}
	+\frac{\|\mathbf{F}\|_\infty^2\Gamma(\alpha)T^{1-\alpha}}{\kappa_\alpha\Gamma(2\alpha -1)}
\biggr).
\]
\end{lemma}
\begin{proof}
For notational convenience, set $z_m(t):= \partial_t^{1-\alpha}v_m(t) \in W_m$.
Taking the inner product of both sides of~\eqref{dmka3} 
with~$z_m(t)$ and integrating by parts with respect to~$x$, we 
obtain
\begin{equation}\label{eq:wh} 
\langle v_m', z_m\rangle+\kappa_\alpha\|\nabla z_m\|^2  
	=\langle g(t),z_m\rangle +\langle \textbf{F}(t)z_m,\nabla z_m\rangle 
    	-\bigl\langle     \nabla\cdot\bigl(\textbf{F}(t)u_{0m}\bigr)
		-\kappa_\alpha\Delta u_{0m},z_m\bigr\rangle\,\omega_\alpha(t).  
\end{equation}
By the Cauchy--Schwarz and  arithmetic--geometric inequalities, one has
\[
\bigl|\langle\textbf{F}(t)z_m,\nabla z_m\rangle\bigr| 
    \leq \|\textbf{F}\|_\infty \|z_m\| \|\nabla z_m\|
    \leq\frac{\kappa_\alpha}{2}\|\nabla z_m\|^2
    +\frac{\|\textbf{F}\|_\infty^2}{2\kappa_\alpha}\|z_m\|^2
\]
and, using Lemma~\ref{lem:Fphi},
\[
\bigl|\bigl\langle
\nabla\cdot(\textbf{F}(t)u_{0m})-\kappa_\alpha\Delta u_{0m},z_m
	\bigr\rangle\bigr|
\le\bigl[\|\mathbf{F}\|_{1,\infty}\|u_{0m}\|_{H^1(\Omega)}
	+\kappa_\alpha\|\triangle u_{0m}\|\bigr]\|z_m\|.
\]
Substituting these bounds into~\eqref{eq:wh} and then applying 
Lemma~\ref{lem:u0m}, we obtain
\begin{equation}\label{M0}
\langle v_m', z_m\rangle+\frac{\kappa_\alpha}{2}\|\nabla z_m\|^2 
	\leq \|g(t)\|\,\|z_m\| 
    +\frac{\|\textbf{F}\|_\infty^2}{2\kappa_\alpha}\|z_m\|^2 
    + \sqrt{C_6}\,\|u_0\|_{H^2(\Omega)}\|z_m\|\,\omega_\alpha(t).
\end{equation}
But $v_m(0)=0$, so $z_m=\partial_t^{1-\alpha}v_m 
={}^C\partial_t^{1-\alpha}v_m=J^\alpha(v_m')$ and thus
$\langle v_m', z_m\rangle=\langle v_m', J^\alpha(v_m')\rangle$. 
Applying~$J^\alpha$ to both sides of~\eqref{M0} and invoking 
Lemma~\ref{lem:pd} to handle the first term, we get
\begin{align}
\|z_m\|^2 + \kappa_\alpha J^\alpha\left(\|\nabla z_m\|^2\right) 
	&\leq 2J^\alpha\left(\|g\| \|z_m\|\right) 
		+\frac{\|\textbf{F}\|_\infty^2}{\kappa_\alpha}
		J^\alpha\left(\|z_m\|^2\right)  \notag\\
	 &\qquad{}+2\sqrt{C_6}\,\|u_0\|_{H^2(\Omega)}
		\,J^\alpha(\|z_m\|\,\omega_\alpha)
\quad\text{for $0<t\le T$.}  \label{M1}
\end{align}
By the Cauchy--Schwarz and arithmetic-geometric mean inequalities,
\[
\begin{aligned}
2J^\alpha(\|g\|\|z_m\|)(t) &= 2 \int_{s=0}^t \frac{(t-s)^{\alpha-1}}{\Gamma(\alpha)}\|g(s)\|\|z_m(s)\|\,ds \\
	&\le2\biggl(\int_{s=0}^t |g(s)|^2\,ds\biggr)^{1/2}
		\biggl(\int_{s=0}^t\frac{(t-s)^{2\alpha-2}}{\Gamma(\alpha)^2}\, 
			\|z_m(s)\|^2\,ds\biggr)^{1/2}\\
	&\le\int_{s=0}^t \|g(s)\|^2\,ds
	+\frac{\Gamma(2\alpha-1)}{\Gamma(\alpha)^2}
	\int_{s=0}^t\omega_{2\alpha-1}(t-s)\|z_m(s)\|^2\,ds,
\end{aligned}
\]
and if $0\le s<t$, then $(t-s)^{\alpha-1} = (t-s)^{1-\alpha} (t-s)^{2\alpha-2} 
\le t^{1-\alpha}(t-s)^{2\alpha-2}$, so 
\[
\frac{\|\textbf{F}\|_\infty^2}{\kappa_\alpha}\,J^\alpha(\|z_m\|^2)(t) 
\le\frac{\|\textbf{F}\|_\infty^2t^{1-\alpha}}{\kappa_\alpha\Gamma(\alpha)}
	\int_{s=0}^t(t-s)^{2\alpha-2}\|z_m(s)\|^2\,ds.
\]
For the final term in~\eqref{M1}, we have
\[
2\sqrt{C_6}\,\|u_0\|_{H^2(\Omega)}\,J^\alpha(\|z_m\|\,\omega_\alpha)
	\le C_6\|u_0\|_{H^2(\Omega)}^2
	+\bigl[J^\alpha(\|z_m\|\omega_\alpha)\big]^2
\]
with
\[
\begin{aligned}
\bigl[J^\alpha(\|z_m\|\omega_\alpha)(t)\bigr]^2 
	&=\biggl(\int_0^t\omega_\alpha(t-s)\|z_m(s)\|\,
		\omega_\alpha(s)\,ds\biggr)^2\\
	&\le\biggl(\int_0^t\omega_\alpha(s)^2\,ds\biggr)
	\biggl(\int_0^t\omega_\alpha(t-s)^2\|z_m(s)\|^2\,ds\biggr)\\
	&=\frac{\Gamma(2\alpha-1)t^{2\alpha-1}}{(2\alpha-1)\Gamma(\alpha)^4}
	\int_0^t\omega_{2\alpha-1}(t-s)\|z_m(s)\|^2\,ds.
\end{aligned}
\]
Hence, \eqref{M1} yields
\begin{multline}\label{M2}
\|z_m(t)\|^2+\kappa_\alpha J^\alpha\left(\|\nabla z_m\|^2\right)(t) 
	\le C_6\|u_0\|_{H^2(\Omega)}^2+\int_0^t\|g(s)\|^2\,ds\\
	+C_7\int_0^t\omega_{2\alpha-1}(t-s)\|z_m(s)\|^2\,ds
\quad\text{for $0<t\le T$.}
\end{multline}
Discard the $\kappa_\alpha$ term and then apply the fractional Gronwall 
inequality (Lemma~\ref{lem:Gronwall}) to get 
\eqref{thmenergy1}. Finally, after substituting the bound~\eqref{thmenergy1} 
into the right-hand side of~\eqref{M2}, it is straightforward to 
deduce~\eqref{thmenergy2}. 
\end{proof}

The next corollary follows easily from Lemma~\ref{thm:energyum1}.

\begin{corollary}\label{cor:energyum1}
\begin{align*}
\max_{0\le t\le T} \|\partial_t^{1-\alpha}v_m(t)\|^2
	&\le C_8\bigl[\|u_0\|^2_{H^2(\Omega)}+\|g\|^2_{L^2}\bigr], \\
\|\partial_t^{1-\alpha}\nabla v_m\|_{L_\alpha^2}^2
	&\le C_9\bigl[\|u_0\|^2_{H^2(\Omega)}+\|g\|^2_{L^2}\bigr]. 
\end{align*}
Here, for $i=8, 9$,  the constants $C_i = C_i(\alpha, T, \kappa_\alpha, \|\textbf{F}\|_{1,\infty})$  blow up as $\alpha\to(1/2)^+$ but are bounded as $\alpha\to 1^-$.
\end{corollary}
Corollary~\ref{cor:energyum1} implies an $L^2(\Omega)$ bound on $u_m(t)$, which we give in Corollary~\ref{cor:umL2}.

\begin{corollary}\label{cor:umL2}
With $C_8$ as in Corollary~\ref{cor:energyum1}, one has
\begin{subequations}
\begin{align}
\|v_m(t)\|^2&\le C_8\omega_{2-\alpha}^2(t)
\bigl[\|u_0\|^2_{H^2(\Omega)}+\|g\|^2_{L^2}\bigr] 
\quad\text{for $0\le t\le T$,}  \label{umL2a} 
\intertext{and}
\|v_m\|_{L_\alpha^2}^2
	&\le \frac{C_8T^{2-\alpha}}{\Gamma(2-\alpha)}
\bigl[\|u_0\|^2_{H^2(\Omega)}+\|g\|^2_{L^2}\bigr]. \label{umL2b}
\end{align}
\end{subequations}
\end{corollary}
\begin{proof}
As $u_m(t)$ is absolutely continuous, we have 
$v_m(t)=(J^1 u_m')(t)=J^{1-\alpha}(J^\alpha u_m')(t)$, where we used 
\cite[Theorem 2.2]{Diet10}. Thus~\cite[Theorem 2.22]{Diet10} can be invoked, 
which yields
\[
v_m(t)=J^{1-\alpha}\partial_t^{1-\alpha}(u_m(t)-u_m(0)) 
	\quad\text{for almost all $t$.}
\]
Set $z_m(t)=\partial_t^{1-\alpha}v_m(t)$, so 
$v_m(t)= J^{1-\alpha}z_m(t)$. Now Lemma~\ref{lem:minsk}~and 
Corollary~\ref{cor:energyum1} give
\begin{align*}
\|v_m(t)\|^2 = \|J^{1-\alpha} z_m(t)\|^2
&\le\omega_{2-\alpha}(t)
	\int_0^t\omega_{1-\alpha}(t-s)\|z_m(s)\|^2\,ds\\
&\le\omega_{2-\alpha}(t) C_8
	\bigl[\|u_0\|^2_{H^2(\Omega)}+\|g\|^2_{L^2(0,T;L^2)}\bigr]
	\int_0^t\omega_{1-\alpha}(t-s)\,ds \\
&=C_8\bigl[\|u_0\|^2_{H^2(\Omega)}+\|g\|^2_{L^2(0,T;L^2)}\bigr]
	\omega_{2-\alpha}^2(t).
\end{align*}
As $u_m(t)$ is continuous, the inequality~\eqref{umL2a} is valid 
for all~$t$.

Next, using~\eqref{umL2a} and the semigroup property 
$\omega_\alpha\ast\omega_\beta=\omega_{\alpha+\beta}$, we  get
\begin{align*}
\|v_m\|_{L_\alpha^2}^2
	&=\max_{0\le t\le T}\int_{s=0}^t\omega_\alpha(t-s)\|v_m(s)\|^2\,ds \\
	&\le C_8\bigl[\|u_0\|^2_{H^2(\Omega)}+\|g\|^2_{L^2(0,T;L^2)}\bigr]
	\max_{0\le t\le T}\int_{s=0}^t\omega_\alpha(t-s)\,
		\omega_{2-\alpha}(s)^2\,ds \\
	&\le C_8\bigl[\|u_0\|^2_{H^2(\Omega)}+\|g\|^2_{L^2(0,T;L^2)}\bigr]
	\biggl(\max_{0\le t\le T}\omega_{2-\alpha}(t)\biggr) 
	\biggl(\max_{0\le t\le T}\omega_2(t)\biggr), 
\end{align*}
which gives~\eqref{umL2b}.
\end{proof}

In the next lemma, we also provide upper bounds 
for~$\{v_m\}_m$ in $W^{1,2}(0,T;L^2)\cap L^2(0,T;H^2)$ and 
$\big\{\partial_t^{1-\alpha}\Delta v_m\big\}_m$ in  $L^2(0,T;L^2)$.  Recall 
that the constant~$\rho_\alpha>0$ was defined in Lemma~\ref{lem:rho}. 

\begin{lemma}\label{thm:energyum2}
Let $m$ be a positive integer. Let $v_m(t)$ be the absolutely continuous 
solution of~\eqref{dmka3} that is guaranteed by Lemma~\ref{lem:existdmk}. Then 
for almost all $t\in [0,T]$, one has
\begin{align}
\|\nabla v_m(t)\|^2 + \kappa_\alpha\rho_\alpha t^{\alpha-1}
	\int_0^t\|\Delta v_m\|^2\,ds 
&\leq\frac{t^{1-\alpha}}{\kappa_\alpha\rho_\alpha}\bigl[
	(C_{10}+C_{11})\|u_0\|_{H^2(\Omega)}^2
	+C_{10}\|g\|_{L^2}^2\bigr], \label{M3} \\
\int_0^t \|v'_m\|^2 \,ds 
&\le(C_{10}+C_{11}) C_R^2\|u_0\|^2_{H^2(\Omega)}+C_{10}\|g\|^2_{L^2}. \label{M7}
\intertext{and}
\int_0^t \|\partial_t^{1-\alpha}\Delta v_m(s)\|^2 \, ds
        &\leq\frac{1}{\kappa_\alpha^2}\bigl[
(C_{10}+C_{11})\|u_0\|^2_{H^2(\Omega)}+C_{10}\|g\|^2_{L^2}\bigr], \label{M6}
\end{align}
where
\[
C_{10}:=3\bigl[1+\|\mathbf{F}\|_{1,\infty}^2(C_8T+C_9\Gamma(\alpha)T^{1-\alpha})\bigr]
\quad\text{and}\quad
C_{11}:=\frac{6C_R^2(\kappa_\alpha^2+\|\mathbf{F}\|_{1,\infty}^2)}%
{(2\alpha-1)\Gamma(\alpha)^2}\,T^{2\alpha-1}.
\]
\end{lemma}
\begin{proof}
Take the inner product of both sides of~\eqref{dmka3} with~$-\Delta v_m\in W_m$ and integrate by parts with respect to $x$ to get 
\begin{multline*}
\frac12\frac{d}{dt} \|\nabla v_m\|^2
	+\kappa_\alpha\langle\partial_t^{1-\alpha}\Delta v_m,\Delta v_m\rangle  
	=\langle\nabla\cdot(\textbf{F}(t)\partial_t^{1-\alpha}v_m), 
		\Delta v_m\rangle
	-\langle g(t),\Delta v_m\rangle \\
	-\bigl\langle\kappa_\alpha\Delta u_{0m}
	-\nabla\cdot(\textbf{F}(t)u_{0m}),
        \Delta v_m\bigr\rangle\,\omega_\alpha(t). 
\end{multline*}
Integrating in time and noting that, by Lemma~\ref{lem:rho},
\[
\rho_\alpha t^{\alpha-1}\int_0^t\|\Delta v_m\|^2\,ds
\leq\int_0^t \langle\partial_s^{1-\alpha}\Delta v_m,\Delta v_m\rangle\,ds,
\]
we obtain
\begin{multline*}
\frac12 \|\nabla v_m(t)\|^2
	+\kappa_\alpha\rho_\alpha t^{\alpha-1}\int_0^t\|\Delta v_m\|^2\,ds
	\leq 3\epsilon\int_0^t\|\Delta v_m\|^2\,ds\\
+\frac{1}{4\epsilon}\int_0^t\Bigl[
	\bigl\|\nabla\cdot\bigl(\textbf{F}(s)\partial_s^{1-\alpha}v_m\bigr)\bigr\|^2
	+\|g(s)\|^2
	+\bigl\|\kappa_\alpha\Delta u_{0m}
		-\nabla\cdot\bigl(\textbf{F}(s)u_{0m}\bigr)\bigr\|^2\omega_\alpha(s)^2
	\Bigr]\,ds,
\end{multline*}
with a free parameter~$\epsilon>0$. Choosing 
$\epsilon=\kappa_\alpha\rho_\alpha t^{\alpha-1}/6$ and recalling 
Lemma~\ref{lem:Fphi} yields
\begin{align*}
\|\nabla v_m(t)\|^2&+\kappa_\alpha\rho_\alpha t^{\alpha-1}
	\int_0^t\|\Delta v_m\|^2\,ds
    \le\frac{3}{\kappa_\alpha\rho_\alpha t^{\alpha-1}}\int_0^t
	\Bigl(\|g(s)\|^2
        +2\kappa_\alpha^2\|\Delta u_{0m}\|^2\omega_\alpha(s)^2\Bigr)\,ds\\
&\qquad{}+\frac{3\|\textbf{F}\|_{1,\infty}^2}%
{\kappa_\alpha\rho_\alpha t^{\alpha-1}}
	\int_0^t \Bigl(\|\partial_s^{1-\alpha}v_m\|_{H^1(\Omega)}^2
        +2\|u_{0m}\|_{H^1(\Omega)}^2\omega_\alpha(s)^2\Bigr)\,ds\\
&\le\frac{t^{1-\alpha}}{\kappa_\alpha\rho_\alpha}\biggl(
	C_{11}\|u_0\|_{H^2(\Omega)}^2+3\|g\|_{L^2}^2
	+3\|\mathbf{F}\|_{1,\infty}^2
	\int_0^t\|\partial_s^{1-\alpha}v_m\|_{H^1(\Omega)}^2\,ds\biggr),
\end{align*}
by Lemma~\ref{lem:u0m}.
Invoking Corollary~\ref{cor:energyum1}, we have
\begin{align}\label{M5}
\int_0^t\|\partial_s^{1-\alpha}v_m\|_{H^1(\Omega)}^2
    &\le\int_0^t\biggl(\|\partial_s^{1-\alpha}v_m\|^2
        +\frac{\omega_\alpha(t-s)}{\omega_\alpha(t)}\,
            \|\partial_s^{1-\alpha}\nabla v_m\|^2\biggr)\,ds\nonumber\\
    &\le\Bigl(
        C_8\,t+C_9\Gamma(\alpha)t^{1-\alpha}\Bigr)
	\bigl[\|u_0\|^2_{H^2(\Omega)}+\|g\|^2_{L^2}\bigr],
\end{align}
and the bound~\eqref{M3} follows.

In a similar fashion, we next take the inner product of both sides 
of~\eqref{dmka3} with~$v'_m\in W_n$ and integrate by parts with respect to~$x$ 
to obtain
\begin{align*}
\|v'_m(t)\|^2+\kappa_\alpha\langle J^{\alpha}\nabla v'_m ,\nabla v'_m\rangle
	&=-\bigl\langle
    \nabla\cdot\bigl(\textbf{F}(t)\partial_t^{1-\alpha}v_m\bigr), v'_m
    \bigr\rangle\\
	&\qquad+\langle g(t),v'_m\rangle+\bigl\langle
	\nabla\cdot\bigl(\textbf{F}(t)u_{0m}\bigr)-\kappa_\alpha \Delta u_m^0,v'_m
    \bigr\rangle\,\omega_\alpha(t)\\
	&\leq3\epsilon\|v'_m(t)\|^2
	+\frac{1}{4\epsilon}\Bigl(
	\|g(t)\|^2
+\bigl\|\nabla\cdot\bigl(\textbf{F}(t)\partial_t^{1-\alpha}v_m\bigr)\bigr\|^2\\
&\qquad{}+\bigl\|\nabla\cdot\bigl(\textbf{F}(t)u_{0m}\bigr)
        -\kappa_\alpha\Delta u_{0m}\bigr\|^2\omega_\alpha^{2}(t)\Bigr). 
\end{align*}
Choosing $\epsilon=1/6$ and invoking Lemma~\ref{lem:Fphi} gives
\begin{align*}
\|v_m'\|^2+2\kappa_\alpha\langle J^\alpha\nabla v_m',\nabla v_m'\rangle
    &\le3\|g(t)\|^2
    +3\|\textbf{F}\|_{1,\infty}^2\|\partial_t^{1-\alpha}v_m\|_{H^1(\Omega)}^2\\
    &\qquad{}+6\bigl(\kappa_\alpha^2+\|\textbf{F}\|_{1,\infty}^2\bigr)
        \|u_{0m}\|_{H^2(\Omega)}^2\omega_\alpha^{2}(t).
\end{align*}
Integrating both sides of the inequality in time and invoking 
Lemma~\ref{lem:pd}, we deduce that
\begin{equation}\label{M4}
\int_0^t \|v'_m\|^2\, ds
    \le3 \|g\|_{L^2(0,T;L^2)}^2
	+3\|\textbf{F}\|_{1,\infty}^2 
	\|\partial_t^{1-\alpha}v_m\|_{L^2(0,t;H^1(\Omega))}^2
	+C_{11}\|u_{0m}\|^2_{H^2(\Omega)}.
\end{equation}
The second result \eqref{M7} now follows from~\eqref{M5}, \eqref{M4} and Lemma~\ref{lem:u0m}.

Using similar arguments, take the inner product of both sides 
of~\eqref{dmka3} with~$-\partial_t^{1-\alpha}\Delta v_m\in W_m$, integrate by parts with respect to~$x$ and note that 
$\partial_t^{1-\alpha}\Delta v_m=J^\alpha\Delta v'_m$
to obtain
\begin{align*}
\langle \nabla v'_m,J^\alpha\nabla v'_m\rangle 
	+\kappa_\alpha\|\partial_t^{1-\alpha}\Delta v_m\|^2 
	&=\langle\nabla\cdot(\textbf{F}(t)\partial_t^{1-\alpha}v_m), 
		\partial_t^{1-\alpha}\Delta v_m\rangle
	-\langle g(t),\partial_t^{1-\alpha}\Delta v_m\rangle \\
	&\quad-\bigl\langle\kappa_\alpha\Delta u_{0m}
	-\nabla\cdot(\textbf{F}(t)u_{0m}),
        \partial_t^{1-\alpha}\Delta v_m\bigr\rangle\,\omega_\alpha(t)\\
    &\leq \frac{\kappa_\alpha}{2}\|\partial_t^{1-\alpha}\Delta v_m\|^2
+\frac{3}{2\kappa_\alpha}\Big(
	\bigl\|\nabla\cdot\bigl(\textbf{F}(t)\partial_t^{1-\alpha}v_m\bigr)\bigr\|^2
	+\|g(t)\|^2\\
&\quad	+\bigl\|\kappa_\alpha\Delta u_{0m}
		-\nabla\cdot\bigl(\textbf{F}u_{0m}\bigr)\bigr\|^2\omega_\alpha^2(t)
		\Big).
\end{align*}
Now integrate in time, invoking 
Lemma~\ref{lem:pd} and using~\eqref{M5}, to deduce that
\begin{align*}
\int_0^t \|\partial_t^{1-\alpha}\Delta v_m(s)\|^2 \, ds
&\leq 
\frac{3}{\kappa_\alpha^2}\int_0^t\Big(
	\bigl\|\nabla\cdot\bigl(\textbf{F}(s)\partial_s^{1-\alpha}v_m\bigr)\bigr\|^2
	+\|g(s)\|^2\\
&\qquad	+\bigl\|\kappa_\alpha\Delta u_{0m}
		-\nabla\cdot\bigl(\textbf{F}u_{0m}\bigr)\bigr\|^2\omega_\alpha^2(s)
		\Big)\,ds\\
&\leq \frac{3}{\kappa_\alpha^2}\int_0^t\Bigl(\|g(s)\|^2
        +2\kappa_\alpha^2\|\Delta u_{0m}\|^2\omega_\alpha(s)^2\Bigr)\,ds\\
    &\qquad +\frac{3\|\textbf{F}\|_{1,\infty}^2}{\kappa_\alpha^2}\int_0^t
    \Bigl(\|\partial_s^{1-\alpha}v_m\|_{H^1(\Omega)}^2
        +2\|u_{0m}\|_{H^1(\Omega)}^2\omega_\alpha(s)^2\Bigr)\,ds,
\end{align*}
and \eqref{M6} now follows by~\eqref{M5} and Lemma~\ref{lem:u0m}, which completes the proof of the 
lemma.
\end{proof}

Inequality~\eqref{M5} may also be derived (with a different constant factor) by applying~\eqref{CR} to~\eqref{M6}.

We remark that the function~$\alpha\mapsto\rho_\alpha$ is monotone increasing
for~$\alpha\in(0,1)$, with $\rho_\alpha\to1$ as~$\alpha\to1$.  
Thus, $\rho_{1/2}<\rho_\alpha<1$ for~$1/2<\alpha<1$, with
$\rho_{1/2}=\sqrt{2\pi/27}=0.48240\ldots$.

\subsection{The classical solution}\label{sec:euweak}
In this section, by using the method of compactness, we show that there is a subsequence of $\{v_m\}_m$ such that the sum of its limit and the inital data 
satisfies equation~\eqref{prob} almost everywhere.
\begin{theorem}\label{the:exit}
Assume that $\alpha\in(1/2,1)$, $u_0\in H^2(\Omega) \cap H_0^1(\Omega)$,  $\textbf{F} \in W^{1,\infty}((0,T)\times\Omega)$ and $g\in L^2(0,T;L^2)$.
Then there exists a unique classical solution of~\eqref{prob}, in the sense of Definition~\ref{def:class}, such that
\begin{equation}\label{eq:ubound}
\sup_{0\leq t\leq T} \|u(t)\|_{H^1(\Omega)}^2
+ \|u'\|_{L^2}^2
+\|\partial_t^{1-\alpha} u\|_{L^2(0,T;H^2)}^2
\leq C_{12}\bigl[\|u_0\|^2_{H^2(\Omega)}+\|g\|^2_{L^2}\bigr],
\end{equation}
where 
\[
C_{12}:= \frac{C_8 T^{2-\alpha}}{\Gamma(2-\alpha)} 
	+ (C_{10}+C_{11})\biggl(1+ C_{\mathrm{R}}^2 + \frac{T^{1-\alpha}}{\kappa_\alpha\rho_\alpha}
	+\frac{1}{\kappa_\alpha^2}\biggr).
\]
\end{theorem}
\begin{proof}
From Corollary~\ref{cor:umL2} and Lemma~\ref{thm:energyum2} we obtain
\begin{equation}\label{eq:vm bounds}
\sup_{0\leq t\leq T} \|v_m(t)\|_{H^1(\Omega)}^2
+\|v_m'\|_{L^2}^2+\|\partial_t^{1-\alpha}v_m\|_{L^2(0,T;H^2)}^2
	\le C_{12}\bigl[\|u_0\|^2_{H^2(\Omega)}+\|g\|^2_{L^2}\bigr],
\end{equation}
which shows that the sequence $\{v_m\}_{m=1}^\infty$ is  
bounded in $L^\infty(0,T;H^1)\cap L^2(0,T;H^2\cap H^1_0)$ and that
the sequence $\{v'_m\}_{m=1}^\infty$ is  bounded in
$L^2(0,T;L^2)$. 
Since the embeddings 
$
H^2 \hookrightarrow H^1 \hookrightarrow L^2
$
are compact, it follows from Lemma~\ref{lem: Aubin} that there exists a subsequence of  
$\{v_m\}_{m=1}^\infty$ (still denoted by  $\{v_m\}_{m=1}^\infty$) such that 
\begin{equation}\label{con:1}
v_m\rightarrow v
\text{ strongly in } C([0,T];L^2)\cap L^2(0,T;H^1).
\end{equation}
Furthermore, from the upper bounds of $\{v_m\}_{m=1}^\infty$ we have
\begin{align}\label{con:3}
&v_m\rightarrow v
\text{ weakly in }L^\infty(0,T;H^1)\cap L^2(0,T;H^2\cap H^1_0)\nonumber\\
\text{and }
&
v'_m\rightarrow v'
\text{ weakly in } L^2(0,T;L^2).
\end{align}
By virtue of Lemma~\ref{lem:L1}, the strong convergence in~\eqref{con:1} implies that 
\begin{equation*}
J^\alpha v_m\rightarrow J^\alpha v
\text{ strongly in } C([0,T];L^2)\cap L^2(0,T;H^1).
\end{equation*}
This, together with Corollary~\ref{cor:energyum1} and~\eqref{M6}, yields
\begin{equation}\label{con:4}
\partial_t J^\alpha v_m\rightarrow \partial_t J^\alpha v
\text{ weakly in } L^\infty(0,T;L^2)\cap L^2(0,T;H^2).
\end{equation}

Multiplying both sides of~\eqref{dmka3} by a test function $\xi\in L^2(0,T;L^2)$, integrating over $(0,T)\times\Omega$
and noting that $\Pi_m$ is a self-adjoint operator on $L^2(\Omega)$, we deduce that
\begin{align*}
\langle v'_m,\xi \rangle_{L^2(0,T;L^2)}
&-\kappa_\alpha\langle\partial_t^{1-\alpha}\Delta v_m ,\xi \rangle_{L^2(0,T;L^2)}
	+\langle\nabla\cdot(\textbf{F}\partial_t^{1-\alpha}v_m),\Pi_m\xi \rangle_{L^2(0,T;L^2)}\nonumber\\
	&= \langle g,\Pi_m\xi \rangle_{L^2(0,T;L^2)}
	-\langle\omega_\alpha\bigl[ 
    \nabla\cdot(\textbf{F}u_{0m})
	-\kappa_\alpha\Delta u_{0m}\bigr],\Pi_m\xi \rangle_{L^2(0,T;L^2)}.
\end{align*}
Now let $m\to\infty$ in this equation and  recall~\eqref{con:3} and \eqref{con:4}. We get
\begin{align}\label{eq:lim1}
\langle u',\xi \rangle_{L^2(0,T;L^2)}
-\kappa_\alpha\langle\partial_t^{1-\alpha}\Delta u &,\xi \rangle_{L^2(0,T;L^2)}
	+\langle\nabla\cdot(\textbf{F}\partial_t^{1-\alpha}u),\xi \rangle_{L^2(0,T;L^2)}
	= \langle g,\xi \rangle_{L^2(0,T;L^2)}
\end{align}
for all $ \xi\in L^2(0,T;L^2)$, where $u:= v + u_0$.
From~\eqref{con:1}--\eqref{con:4}, we have 
\[
u'\in L^2(0,T;L^2)\quad \text{and}\quad 
\partial_t^{1-\alpha} u\in L^2(0,T;H^2).
\]
Hence, it follows from~\eqref{eq:lim1} that $u$ satisfies~\eqref{prob} a.e. in $(0,T)\times \Omega$.

Taking the limit as $m\to\infty$ in~\eqref{eq:vm bounds}, we 
obtain~\eqref{eq:ubound}. The uniqueness of the solution~$u$ follows 
from~\eqref{eq:ubound}, which completes the proof of the theorem.
\end{proof}
\begin{remark}
It follows from the uniqueness in Theorems~\ref{the:exit0} and~\ref{the:exit} 
that the mild solution will become the classical solution when 
$\alpha\in(1/2,1)$ and $u_0\in H^2(\Omega)\cap H^1_0(\Omega)$. Furthermore, the 
continuous dependence of both the mild  and classical solutions on the initial 
data $u_0$ follows from~\eqref{eq:ubound0} and~\eqref{eq:ubound}. 
\end{remark}

\section{Regularity of the classical solution}\label{sec:regclassical}
Recall that $1/2 < \alpha <1$ 
and that in general $C = C(\Omega, \kappa_\alpha, \textbf{F}, T)$. From Lemma~\ref{lem:f} onwards, we allow $C = C(\Omega, \kappa_\alpha, \textbf{F}, T, q)$, where $q$ appears in the statements of our results below.

From Theorem~\ref{the:exit}, for almost every 
$(t,x)\in (0,T)\times\Omega$ the solution $u(t,x)$ satisfies~\eqref{prob}. 
Using the identity $\partial_t^{1-\alpha}u=(J^\alpha u)'(t)
=(J^\alpha u')(t)+u(0)\omega_\alpha(t)$, we rewrite~\eqref{prob} as
\begin{equation}\label{eq:prob2}
u'-\nabla\cdot(\kappa_\alpha\nabla J^\alpha u'-\textbf{F}J^{\alpha}u')
	=g(t)+ \nabla\cdot\bigl[\kappa_\alpha\nabla u_{0}-\textbf{F}(t)u_{0}\bigr] 
	\omega_\alpha(t).
\end{equation}
From this equation and the fact that $J^\alpha\phi(0) = 0$ for any 
function $\phi\in L^2(0,T;L^2)$, we deduce that $u'(t)=O(t^{\alpha-1})$ 
when $t$ is close to $0$. By letting $z(t,x):= tu'(t,x)$,  we have $z(0) = 0$. 
The regularity of $z$ is examined in the following lemma.

\begin{lemma} \label{lem:reg z}
Assume that $\int_0^T\|tg'(t)\|^2\,dt$ is finite. Then, the function $z$ 
defined above satisfies
\begin{align}\label{eq:zbound}
\sup_{0\leq t\leq T} \|z(t)\|_{H^1(\Omega)}^2
+\|z'\|_{L^2}^2
+\|\partial_t^{1-\alpha} z\|_{L^2(0,T;H^2)}^2
\leq C_{13}\bigl(\|u_0\|_{H^2(\Omega)}^2+\|g\|_{L^2}^2
	+\|g_1\|_{L^2}^2\bigr) 
\end{align}
for some constant $C_{13}$. 
\end{lemma}
\begin{proof}
For any $t>0$, multiplying both sides of~\eqref{eq:prob2} by $t$ and using the 
elementary identity
\begin{align}\label{eq:inter}
t(J^\alpha u')(t)&=(J^\alpha z)(t)+\alpha(J^{\alpha+1} u')(t)
	=(J^\alpha z)(t)+\alpha\big((J^{\alpha}u)(t)-u_0\omega_{\alpha+1}(t)\big)\\
	&=(J^\alpha z)(t)+\alpha(J^{\alpha}u)(t)-u_0t\omega_\alpha(t),\nonumber
\end{align}
we obtain a differential equation for~$z$: 
\[
z - \nabla\cdot(J^{\alpha}\kappa_\alpha\nabla z - \textbf{F}J^{\alpha}z)
=tg(t)+\alpha\nabla\cdot\big(\kappa_\alpha\nabla J^\alpha u
	-\mathbf{F}J^\alpha u\big).
\]
Differentiating both sides of this equation with respect to $t$ and noting 
that $J^{\alpha}z' = \partial_t^{1-\alpha}z$, we have
\begin{equation}\label{eq:prob3}
z' - \nabla\cdot(\partial_t^{1-\alpha}\kappa_\alpha\nabla z 
	- \textbf{F}\partial_t^{1-\alpha} z)
= \bar{G}(t,x),
\end{equation}
where
\[
\bar{G}:=g+tg'+\alpha\nabla\cdot\bigl(
	\kappa_\alpha\nabla\partial_t^{1-\alpha}u
	-\mathbf{F}'(t)J^\alpha u-\mathbf{F}(t)\partial_t^{1-\alpha}u\bigr).
\]
Applying Lemma~\ref{lem:Fphi} and letting $g_1(t)=tg'(t)$, we find that
\[
\|\bar{G}\|_{L^2}^2\le 4\Bigl(\|g\|_{L^2}^2+\|g_1\|_{L^2}^2
	+\alpha^2(\kappa_\alpha+\|\mathbf{F}\|_{1,\infty})^2
	\|\partial_t^{1-\alpha}u\|_{L^2(0,T;H^2)}^2
	+\alpha^2\|\mathbf{F}'\|_{1,\infty}^2\|J^\alpha u\|_{L^2(0,T,H^1)}^2\Bigr),
\]
with Lemma~\ref{lem: priori weak1}~and Theorem~\ref{the:exit} implying that
\[
\|\bar{G}\|_{L^2}^2\le C\bigl(\|u_0\|_{H^2(\Omega)}^2+\|g\|_{L^2}^2
	+\|g_1\|_{L^2}^2\bigr) \ \text{ for some constant }C.
\]
Thus, applying Theorem~\ref{the:exit} to equation~\eqref{eq:prob3} with initial 
data $z(0) = 0$, we deduce the bound~\eqref{eq:zbound}.
\end{proof}
From Theorem~\ref{the:exit}, for almost every $(t,x)\in (0,T)\times\Omega$ we 
have the identity
\begin{equation}\label{eq:ProSing}
u_t -\nabla\cdot(\partial_t^{1-\alpha}\kappa_\alpha\nabla u)) = f,
\end{equation}
where $f:= g-\nabla\cdot(\textbf{F}\partial_t^{1-\alpha}u)\in L^2(0,T;H^1)$. 
The regularity of solutions to problem~\eqref{eq:ProSing} subject to the initial 
condition $u_0\in H^2(\Omega) \cap H_0^1(\Omega)$ was studied 
in~\cite{Bill2010}. In order to apply~\cite[Theorem 5.7]{Bill2010}, we need at 
least an upper bound for $\int_0^t s\|f'(s)\|ds $ which is proved in the 
following lemma. Here and subsequently, notation such as $f'$ and $f^{(j)}$ 
indicates time derivatives, and we denote higher-order fractional derivatives by
$\partial^{j-\alpha}_tu:=\partial_t^{j-1}\partial_t^{1-\alpha}u
=(J^\alpha u)^{(j)}$ for~$j\in\{1,2,3,\ldots\}$ and~$0<\alpha<1$.  
\begin{lemma}\label{lem:f}
Let $u$ be the solution of~\eqref{prob}~and 
$f:=g-\nabla\cdot(\textbf{F}\partial_t^{1-\alpha}u)$. 
Then, for $q\in \{0,1,2,\dots\}$, $\textbf{F}\in W^{q,\infty}(0,T;L^2(\Omega))$, and for any $t\in(0,T]$,
there is a constant $C = C(\Omega, \kappa_\alpha, \textbf{F}, q)$ such that
\begin{equation}\label{eq:f1}
\int_0^t s^{2q}\|f^{(q)}(s)\|^2\,ds
	\le C\biggl(\|u_0\|_{H^2(\Omega)}^2+\sum_{j=0}^q
		\int_0^t s^{2j}\|g^{(j)}(s)\|^2\,ds\biggr).
\end{equation}
\end{lemma}
\begin{proof}
Inequality~\eqref{eq:f1} holds for $q=0$ by virtue of~\eqref{eq:ubound} (with~$t$ playing the role of~$T$) because
\[
\|f(t)\|^2\le C\bigl(\|g(t)\|^2
	+\|\partial_t^{1-\alpha}u\|_{H^1(\Omega)}\bigr)^2.
\]
For the case~$q=1$, we note first that
\begin{align*}
t^2\|f'(t)\|^2&=t^2\bigl\|
g'(t)-\nabla\cdot\bigl(\mathbf{F}'(t)\partial_t^{1-\alpha}u
+\mathbf{F}(t)\partial_t^{2-\alpha}u\bigr)\bigr\|^2\\
	&\le Ct^2\bigl(\|g'(t)\|^2
	+\|\partial_t^{1-\alpha}u\|_{H^1(\Omega)}^2
+\|\partial_t^{2-\alpha}u\|_{H^1(\Omega)}^2\bigr).
\end{align*}
By~\eqref{eq:inter} we have
$(J^\alpha z)(t)=t(J^\alpha u')(t)-\alpha(J^{\alpha+1}u')(t)$,
and differentiating with respect to~$t$ gives
\[
\partial_t^{1-\alpha}z=t(J^\alpha u')'(t)-(\alpha-1)(J^\alpha u')(t) 
	=t\partial_t^{2-\alpha}u-(\alpha-1)\partial_t^{1-\alpha}u,
\]
where we used the identities 
$(J^\alpha u')(t)=\partial_t^{1-\alpha}u-u_0\omega_\alpha(t)$~and 
$(\alpha-1)\omega_\alpha(t)=t\omega_{\alpha-1}(t)$.  Thus,
\begin{equation}\label{eq:uz}
t\partial_t^{2-\alpha} u = 
\partial_t^{1-\alpha} z + (\alpha-1)\partial_t^{1-\alpha} u.
\end{equation}
Hence, by Theorem~\ref{the:exit}~and Lemma~\ref{lem:reg z} (with~$t$ again playing the role of~$T$),
\begin{equation}\label{eq:fu1}
\int_0^t s^2\|\partial_s^{2-\alpha}u\|_{H^2(\Omega)}^2\,ds
	\le C\biggl(\|u_0\|_{H^2(\Omega)}^2+\int_0^t\bigl[\|g(s)\|^2+s^2\|g'(s)\|^2
	\bigr]\,ds\biggr),
\end{equation}
implying that the desired inequality \eqref{eq:f1} holds for~$q=1$. 

Multiply both sides of~\eqref{eq:uz} by $t$ and then differentiate with respect 
to~$t$, obtaining
\begin{align}\label{eq:uz2}
t^2\partial_t^{3-\alpha} u = 
\partial_t^{1-\alpha} z + t\partial_t^{2-\alpha} z 
+  (\alpha-1)\partial_t^{1-\alpha} u 
+ (\alpha-3)t\,\partial_t^{2-\alpha} u.
\end{align}
Since $z$ satisfies~\eqref{eq:prob3} --- an equation similar 
to~\eqref{proba} but with a different source~$\bar G$ and with $z(0)=0$ --- we 
get an estimate for $z$ corresponding to~\eqref{eq:fu1}: 
\begin{equation*}\label{eq:fz1}
\int_0^t s^2\|\partial_s^{2-\alpha} z\|^2_{H^2(\Omega)}\,ds
	\le C\int_0^t\bigl[\|\bar G(s)\|^2+s^2\|\bar G'(s)\|^2\bigr]\,ds.
\end{equation*} 
This inequality, together with~\eqref{eq:ubound}, \eqref{eq:zbound}, \eqref{eq:fu1}~and
\eqref{eq:uz2}, yields
\begin{equation}\label{eq:fu2}
\int_0^t s^4\|\partial_t^{3-\alpha} u\|^2_{H^2(\Omega)}\,ds
	\le C\biggl(\|u_0\|_{H^2(\Omega)}^2+\int_0^t\bigl[\|g(s)\|^2+s^2\|g'(s)\|^2
	+s^4\|g''(s)\|^2\bigr]\,ds\biggr),
\end{equation}
which implies the desired inequality \eqref{eq:f1} for~$q=2$.

The general case follows by iterating the arguments above; cf.~\cite{MMAK2018}.
\end{proof}

We can now prove regularity estimates for the classical solution $u$.

\begin{theorem}\label{th: regu}
Let $g_j(t):=t^jg^{(j)}(t)$ for $j=1,2,3,\dots$ For $q\in \{1,2,3,\dots\}, \textbf{F}\in W^{q,\infty}(0,T;L^2(\Omega))$ and for any 
$t\in(0,T]$, 
\[
t^q\|\Delta u^{(q)}(t)\|\leq Ct^{-(\alpha-1/2)}\biggl(\|u_0\|_{H^2(\Omega)}
	+\sum_{j=0}^{q+1}\|g_j\|_{L^2}\biggr)
\]
and
\[
t^{q}\| u^{(q)}(t)\|\le Ct^{1/2}\biggl(\|u_0\|_{H^2(\Omega)}
	+\sum_{j=0}^q\|g_j\|_{L^2}\biggr).
\]
\end{theorem}
\begin{proof}
%

By~\eqref{eq:ProSing}, it follows from \cite[Theorem~4.4]{Bill2010} with 
$r=2$~and $\nu=\alpha$, and from \cite[Theorem~5.6]{Bill2010} with 
$r=0$, $\mu=2$~and $\nu=\alpha$, that
\[
t^q\|\Delta u^{(q)}(t)\|\leq C
\bigg(\|u_0\|_{H^2(\Omega)} +t^{-\alpha}
\sum_{j=0}^{q+1}\int_0^t s^j\|f^{(j)}(s)\|\,ds\bigg).
\]
Similarly, from~\cite[Theorem 4.4]{Bill2010} with $r=2$~and $\nu=\alpha$, and 
from \cite[Theorem~5.4]{Bill2010} with $r=\mu=0$,
and
\[
t^q\| u^{(q)}(t)\|\leq C
\bigg(t^\alpha\|u_0\|_{H^2(\Omega)}+
\sum_{j=0}^q\int_0^t s^j\|f^{(j)}(s)\|\,ds\bigg).
\]
The theorem follows by Lemma~\ref{lem:f} since
$\int_0^ts^j\|f^{(j)}(s)\|\,ds\le 
t^{1/2}\bigl(\int_0^ts^{2j}\|f^{(j)}(s)\|^2\,ds\bigr)^{1/2}$.
\end{proof}

\begin{corollary}\label{cor:reg}
Let $\eta>1/2$. If $\|g^{(j)}(t)\| \le Mt^{\eta-1-j}$ for $0\le j\le q+1, \textbf{F}\in W^{q,\infty}(0,T;L^2(\Omega))$~and
$t\in(0,T]$, then
\[
t^q\|\Delta u^{(q)}(t)\| \le C\bigl(t^{-(\alpha-1/2)}\|u_0\|_{H^2(\Omega)}
	+Mt^{\eta-\alpha}\bigr)
\quad\text{and}\quad
t^q\|u^{(q)}(t)\| \le C\bigl(t^{1/2}\|u_0\|_{H^2(\Omega)}
	+Mt^\eta\bigr).
\]
\end{corollary}
\begin{proof}
The assumption on~$g$ ensures that $\|g_j\| \le Mt^{\eta-1/2}$.
\end{proof}

The alternative and longer analysis in \cite[Theorems 6.2~and 6.3]{MMAK2018} 
shows that these bounds can be improved to
\[
t^q\|\Delta u^{(q)}(t)\|\le 
C\bigl(\|u_0\|_{H^2(\Omega)}+Mt^{\eta-\alpha}\bigr)
\quad\text{and}\quad
t^q\|u^{(q)}(t)\|\le C\bigl(t^\alpha\|u_0\|_{H^2(\Omega)}+Mt^\eta\bigr),
\]
for any $\alpha\in(0,1)$~and $\eta>0$.  

\providecommand{\href}[2]{#2}
\providecommand{\arxiv}[1]{\href{http://arxiv.org/abs/#1}{arXiv:#1}}
\providecommand{\url}[1]{\texttt{#1}}
\providecommand{\urlprefix}{URL }

\end{document}